\documentclass[preprint,12pt]{elsarticle}

\usepackage{mathrsfs,xcolor}
\usepackage{epsfig, graphicx}
\usepackage{latexsym,amsfonts,amsbsy,amssymb,ulem}
\usepackage{amsmath,amsthm}

\makeatletter

\@addtoreset{equation}{section}

\allowdisplaybreaks

\newtheorem{theorem}{Theorem}[section]
\newtheorem{lemma}{Lemma}[section]
\newtheorem{corollary}{Corollary}[section]
\newtheorem{remark}{Remark}[section]
\newtheorem{algorithm}{Algorithm}[section]

\journal{Journal of Computational Physics}

\begin{document}

\begin{frontmatter}



\title{A Multilevel Correction Method for Interior Transmission Eigenvalue Problem}


\author[lsec]{Hehu Xie}
\ead{hhxie@lsec.cc.ac.cn}

\author[fudan]{Xinming Wu\corref{cor1}}
\ead{wuxinming@fudan.edu.cn}

\cortext[cor1]{Corresponding author}

\address[lsec]{LSEC, ICMSEC, Academy of Mathematics and System Sciences,
Chinese Academy of Sciences, Beijing 100190, P.R. China.}
\address[fudan]{The Key Laboratory of Mathematics for Nonlinear Sciences,
School of Mathematical Sciences, Fudan University, Shanghai 200433, P.R. China.}

\begin{abstract}
In this paper, we give a numerical analysis for the transmission eigenvalue problem by the
finite element method. A type of multilevel correction method is proposed to
solve the transmission eigenvalue problem. The multilevel correction method
can transform the transmission eigenvalue solving in the finest finite element space
to a sequence of linear problems and some transmission eigenvalue solving
in a very low dimensional spaces. Since the main computational work is to solve the sequence of
linear problems, the multilevel correction method improves the overfull efficiency
of the transmission eigenvalue solving. Some numerical examples are
provided to validate the theoretical results and the efficiency of the proposed numerical scheme.
\end{abstract}


\end{frontmatter}



\section{Introduction}
Recently, many researchers are interested in the transmission eigenvalue problem \cite{BonnetChesnelHaddar,CakoniCayorenColton,CakoniColtonMonkSun,CakoniGintidesHaddar,CakoniHaddar,ColtonKress,
ColtonMonkSun,ColtonPaivarintaSykvester,Kirsch,PaivarintaSylvester}.
The transmission eigenvalue problem arises in the study of the inverse scattering
for inhomogeneous media which not only has theoretical
importance \cite{CakoniHaddar,ColtonKress}, but also can be used to estimate the properties
of the scattering material  \cite{CakoniCayorenColton,CakoniGintidesHaddar,Sun_1}
 since they can be determined from the scattering data.
In the past few years, significant progress of the existence of transmission
eigenvalues and applications has been made. We refer the readers to the recent
papers \cite{BonnetChesnelHaddar,ChesnelCiarlet,CakoniHaddar}.

Meanwhile, there are also many papers to give the numerical treatment for the transmission eigenvalue problem  and the associated interior transmission problem
\cite{AnShen,BonnetChesnelCiarlet,ChesnelCiarlet,ColtonMonkSun,HsialLiuSunXu,
JiSunTurner,JiSunXie,Sun,WuChen}. But there are few papers
providing the corresponding theoretical analysis for their numerical methods
due to the difficulty that the problem is neither elliptic nor self-adjoint.
The paper \cite{JiSunXie} presents an accurate error estimate of the eigenvalue and eigenfunction
 approximations for the Helmhotz transmission eigenvalue problem based on the
iterative methods (bisection and secant) from \cite{Sun}.
The first aim of this paper is to give a theoretical analysis of the finite element method
for the transmission eigenvalue problem with the inhomogeneous media.

In the past few years, a new type of multilevel correction method is proposed to solve the eigenvalue
problem \cite{LinXie_Steklov,LinXie,Xie_Nonconforming}.
In the multilevel correction scheme, the solution on the finest mesh can be reduced to a
series of solutions of the eigenvalue problem in a very low dimensional
 space and a series of solutions of the boundary value problem on the multilevel meshes.
 This multilevel correction method gives a way to construct
a type of multigrid scheme for the  eigenvalue problem \cite{JiSunXie,Xie_IMA,Xie_JCP}.
The second aim of this paper is to propose a multilevel correction method for the
transmission eigenvalue problem based on the obtained error estimate results.

The rest of this paper is organized as follows. In Section \ref{Transmission_Eigenvalue_Problem},
we introduce the transmission eigenvalue problem and the corresponding theoretical results about
the eigenvalue distribution.
The finite element method and the corresponding error estimates are given in Section  \ref{FEM_Method}.
Section \ref{Multilevel_Correction_Method} is devoted to introducing a type of multilevel correction method for the
transmission eigenvalue problem. In Section \ref{Numerical_Results},
four examples are presented to validate the theoretical results and
 the efficiency of the proposed numerical methods.
Some concluding remarks are given in the last section.

\section{Transmission eigenvalue problem}\label{Transmission_Eigenvalue_Problem}
First, we introduce some notation and the transmission eigenvalue problem. The letter
$C$ (with or without subscripts) denotes a generic
positive constant which may be different at its different occurrences through the paper.
For convenience, the symbols $\lesssim$, $\gtrsim$ and $\approx$
will be used in this paper. Notations $x_1\lesssim y_1, x_2\gtrsim y_2$
and $x_3\approx y_3$, mean that $x_1\leq C_1y_1$, $x_2 \geq c_2y_2$
and $c_3x_3\leq y_3\leq C_3x_3$ for some constants $C_1, c_2, c_3$
and $C_3$ that are independent of mesh sizes.

In this paper, we are concerned with the transmission eigenvalues corresponding to
the scattering of acoustic waves by a bounded
simply connected inhomogeneous medium $\Omega\subset \mathcal{R}^d$ ($d=2, 3$).
The transmission eigenvalue problem is to find $k\in \mathcal{C}$,
 $(w,v)\in H^1(\Omega)\times H^1(\Omega)$ such that
\begin{equation}\label{Problem}
\left\{
\begin{array}{rcll}
-{\rm div}\big(A\nabla w\big)&=&k^2n(x)w,&{\rm in}\ \Omega,\\
-\Delta v &=&k^2v, &{\rm in}\ \Omega,\\
w-v&=&0,&{\rm on}\ \partial \Omega,\\
\frac{\partial w}{\partial \nu_{A}}-\frac{\partial v}{\partial\nu}&=&0,&{\rm on}\ \partial \Omega,
\end{array}
\right.
\end{equation}
where $\nu$ is the unit outward normal to the boundary $\partial \Omega$. There exists
a real number $\gamma>1$ such that the symmetric matrix $A(x)$ and the index of refraction $n(x)$ satisfy that
\begin{equation}\label{cond-A-n}
\xi\cdot A\xi>\gamma |\xi|^2\ \ \ \forall \xi\in \mathcal{R}^d, \quad
n(x)>\gamma, \  \ a.e. \ \ \mbox{in}\  \Omega.
\end{equation}
Values of $k$ such that there exists a nontrivial solution $(w,v)$
to (\ref{Problem}) are called transmission eigenvalues.

\begin{remark}\label{Remark_Condition}
As in \cite{BonnetChesnelCiarlet, Ciarlet2}, the numerical method and analysis can be extended to the case
that there exists a real number $0<\gamma<1$ such that the symmetric matrix $A(x)$ and the index of
refraction  $n(x)$ satisfy that
\begin{equation}\label{cond-A-n-2}
0<\xi\cdot A\xi<\gamma |\xi|^2\ \ \ \forall \xi\in \mathcal{R}^d, \quad
0<n(x)<\gamma, \  \ a.e. \ \ \mbox{in}\  \Omega.
\end{equation}
\end{remark}

Obviously, the eigenvalue problem (\ref{Problem}) can be transformed into the following version:
Find $\lambda\in \mathcal{C}$, $(u,w)\in H^1(\Omega)\times H^1(\Omega)$ such that
\begin{equation}\label{Problem_2}
\left\{
\begin{array}{rcll}
-{\rm div}\big(A\nabla w\big)+n(x)w&=&\lambda n(x)w,&{\rm in}\ \Omega,\\
-\Delta v+v&=&\lambda v, &{\rm in}\ \Omega,\\
w-v&=&0,&{\rm on}\ \partial \Omega,\\
\frac{\partial w}{\partial \nu_{A}}-\frac{\partial v}{\partial\nu}&=&0,&{\rm on}\ \partial \Omega,
\end{array}
\right.
\end{equation}
where $\lambda=k^2+1$. In the following of this paper, we mainly consider this eigenvalue problem.
There are some papers \cite{BonnetChesnelHaddar,CakoniGintidesHaddar,CakoniHaddar,Kirsch,PaivarintaSylvester}
being concerned with the distribution of the eigenvalues for the eigenvalue problem (\ref{Problem_2}).

In this paper, in order to give the analysis, we define the function spaces
$\mathbb V$ and $\mathbb W$ as follows
\begin{eqnarray}
\mathbb V&:=&\Big\{\mathbf\Psi
:=(\varphi,\psi)\in H^1(\Omega)\times H^1(\Omega)\ |\ \varphi-\psi\in H_0^1(\Omega)\Big\},\\
\mathbb W&:=&L^2(\Omega)\times L^2(\Omega)
\end{eqnarray}
equipped with the norms
\begin{eqnarray*}
\|\mathbf \Psi\|_{\mathbb V}=\Big(\|\varphi\|_1^2+\|\psi\|_1^2\Big)^{1/2}\ \ {\rm and}\ \
\|\mathbf \Psi\|_{\mathbb W}=\Big(\|\varphi\|_0^2+\|\psi\|_0^2\Big)^{1/2},
\end{eqnarray*}
respectively, where $\mathbf \Psi=(\varphi,\psi)\in\mathbb{V}$.

For the simplicity of notation, we define two sesquilinear forms
\begin{eqnarray}
a(\mathbf U,\mathbf \Psi)&=&\big(A\nabla w,\nabla \varphi\big)+\big(n(x)w,\varphi\big)
-\big(\nabla v,\nabla \psi\big)-\big(v,\psi\big),\\
b(\mathbf U,\mathbf \Psi)&=&\big(n(x)w,\varphi\big)-\big(v,\psi\big),
\end{eqnarray}
where $\mathbf U=(w,v), \mathbf \Psi=(\varphi,\psi)\in \mathbb{V}$.
The associated variational form for (\ref{Problem_2}) can be defined as follows:
Find $(\lambda,\mathbf U) \in \mathcal{C}\times\mathbb V$ such that $\|\mathbf U\|_{\mathbb W}=1$ and
\begin{eqnarray}\label{Problem_Weak}
a(\mathbf U,\mathbf \Psi)&=&\lambda b(\mathbf U,\mathbf \Psi), \ \ \ \forall \mathbf \Psi\in \mathbb V.
\end{eqnarray}
Then the corresponding adjoint eigenvalue problem is:
Find $(\lambda,\mathbf U^*) \in \mathcal{C}\times\mathbb V$ such that $\|\mathbf U^*\|_{\mathbb W}=1$ and
\begin{eqnarray}\label{Eigenvalue_Problem_Weak_Adjoint}
a(\mathbf \Psi,\mathbf U^*)&=&\lambda b(\mathbf \Psi, \mathbf U^*), \ \ \ \forall \mathbf \Psi\in \mathbb V.
\end{eqnarray}

In order to analyze the properties of the eigenvalue problem (\ref{Problem_Weak}), we introduce
the so-called $\mathbb T$-coercivity (inf-sup condition) for the bilinear form $a(\cdot,\cdot)$
(see, e.g., \cite{BonnetChesnelHaddar,BonnetChesnelCiarlet,ChesnelCiarlet,BonnetCiarletZwolf,Ciarlet2}).
In this paper, the notation $\mathbb T$ denotes an isomorphic operator from $\mathbb V$ to $\mathbb V$ which is
defined as follows
\begin{eqnarray}
\mathbb T\mathbf \Psi&=&(\varphi, 2\varphi-\psi),\ \ \ \ \forall \mathbf \Psi\in \mathbb V.
\end{eqnarray}
Similarly to \cite{BonnetChesnelHaddar,BonnetChesnelCiarlet,ChesnelCiarlet}, in order to give the eigenvalue distribution of (\ref{Problem_2}),
we also state the following $\mathbb T$-coercivity properties (inf-sup conditions).
\begin{theorem}\label{Inf_Sup_Theorem}
The bilinear forms $a(\cdot,\cdot)$ and $b(\cdot,\cdot)$ have the following inf-sup conditions
($\mathbb T$-coercivities):
\begin{eqnarray}
\inf_{0\neq \mathbf\Phi\in\mathbb V}\sup_{0\neq \mathbf \Psi\in \mathbb V}
\frac{a(\mathbf\Phi,\mathbf\Psi)}{\|\mathbf\Phi\|_{\mathbb V}\|\mathbf \Psi\|_{\mathbb V}}\geq \mu_a,\label{Inf-sup_left}\\
\inf_{0\neq \mathbf\Phi\in\mathbb V}\sup_{0\neq \mathbf \Psi\in \mathbb V}
\frac{a(\mathbf\Psi,\mathbf\Phi)}{\|\mathbf\Psi\|_{\mathbb V}\|\mathbf \Phi\|_{\mathbb V}}\geq \mu_a,\label{Inf-sup_right}
\end{eqnarray}
and
\begin{eqnarray}
\inf_{0\neq \mathbf\Phi\in\mathbb W}\sup_{0\neq \mathbf \Psi\in \mathbb W}
\frac{b(\mathbf\Phi,\mathbf\Psi)}{\|\mathbf\Phi\|_{\mathbb W}\|\mathbf \Psi\|_{\mathbb W}}\geq \mu_b,\label{Inf-sup_left_b}\\
\inf_{0\neq \mathbf\Phi\in\mathbb W}\sup_{0\neq \mathbf \Psi\in \mathbb W}
\frac{b(\mathbf\Psi,\mathbf\Phi)}{\|\mathbf\Psi\|_{\mathbb W}\|\mathbf \Phi\|_{\mathbb W}}\geq \mu_b,\label{Inf-sup_right_b}
\end{eqnarray}
for some positive constants $\mu_a$ and $\mu_b$.
\end{theorem}
\begin{proof}
From the conditions of the matrix $A$ and the refraction index $n(x)$, the following estimates hold
\begin{eqnarray}
a(\mathbf \Psi,\mathbb T\mathbf \Psi)&=&(A\nabla\varphi,\nabla\varphi)+(n(x)\varphi,\varphi)
-(\nabla\psi,\nabla(2\varphi-\psi))\nonumber\\
&&\ \ \ \ -(\psi,(2\varphi-\psi))\nonumber\\
&\geq&\gamma \|\varphi\|_1^2+\|\psi\|_1^2 - 2(\nabla\psi,\nabla\varphi) - 2(\psi,\varphi)\nonumber\\
&\geq& \Big(\gamma-\frac{1}{\delta}\Big)\|\varphi\|_1^2+\Big(1-\delta\Big)\|\psi\|_1^2\nonumber\\
&\geq& C\|\mathbf\Psi\|_{\mathbb V}^2.
\end{eqnarray}
Since $\gamma>1$, we can choose $\delta \in \Big(\frac{1}{\gamma},1\Big)$ such that $a(\cdot,\cdot)$
 is $\mathbb T$-coercive which means there exists a positive constant $\mu_a$ such that the desired result (\ref{Inf-sup_left})
 holds. In the same way, we can also prove the result (\ref{Inf-sup_right}).

Similarly, it is easy to prove that
\begin{eqnarray}\
b(\mathbf \Psi,\mathbb T\mathbf \Psi)&=&(n(x)\varphi,\varphi)-(\psi,2\varphi-\psi)\nonumber\\
&\geq& \Big(\gamma-\frac{1}{\delta}\Big)\|\varphi\|_0^2+(1-\delta)\|\psi\|_0^2.
\end{eqnarray}
Since $\gamma>1$, we can choose $\delta \in \Big(\frac{1}{\gamma},1\Big)$ such that $b(\cdot,\cdot)$
 is $\mathbb T$-coercive in $\mathbb W\times \mathbb W$ which means the inf-sup conditions (\ref{Inf-sup_left_b})
 and (\ref{Inf-sup_right_b}) hold for some positive constant $\mu_b$.
 \end{proof}

We introduce the operators $\mathbb K,\ \mathbb K_*\in \mathcal{L}(\mathbb V)$ defined by the equations
\begin{eqnarray}
a(\mathbb K \mathbf \Phi,\mathbf\Psi)=b(\mathbf \Phi,\mathbf\Psi),\ \
a(\mathbf \Psi, \mathbb K_* \mathbf \Phi)=b(\mathbf\Psi,\mathbf \Phi),\ \ \ \ \ \forall\mathbf \Phi, \mathbf\Psi\in\mathbb V.
\end{eqnarray}
From Theorem \ref{Inf_Sup_Theorem}, it is easy to know the operator $\mathbb K$ and $\mathbb K_*$ are linear bijective operators.
 Then the eigenvalue problem (\ref{Problem_Weak}) can be written as an operator form
 for $\lambda\neq 0$ (denoting $\mu:=\lambda^{-1}$):
\begin{eqnarray}
\mathbb K\mathbf U=\mu \mathbf U,
\end{eqnarray}
with
\begin{eqnarray}
\mathbb K_* \mathbf U^*=\bar{\mu}\mathbf U^*
\end{eqnarray}
for the adjoint eigenvalue problem (\ref{Eigenvalue_Problem_Weak_Adjoint}).
The $\mathbb T$-coercivity conditions (\ref{Inf-sup_left})-(\ref{Inf-sup_right})
and (\ref{Inf-sup_left_b})-(\ref{Inf-sup_right_b})
guarantee that every eigenvalue $\lambda$ is nonzero. From (\ref{Inf-sup_left})
and (\ref{Inf-sup_right}) and the compact embedding theorem
of Sobolev spaces, it is well known that the operators $\mathbb K$ and $\mathbb K_*$ are compact.
 Thus the spectral theory for compact
operators  gives us a complete characterization of the eigenvalue problem (\ref{Problem_Weak}).

There is a countable set of eigenvalues of (\ref{Problem_Weak}). Let $\lambda$ be an
eigenvalue of problem (\ref{Problem_Weak}). There exists a smallest integer $\alpha$ such that
\begin{eqnarray}
{\rm Null}((\mathbb K-\mu)^{\alpha})={\rm Null}((\mathbb K-\mu)^{\alpha+1}),
\end{eqnarray}
where ${\rm Null}$ denotes the null space and we use the notation $\mu=\lambda^{-1}$.
Let
\begin{eqnarray*}
M(\lambda)=M_{\lambda,\mu}={\rm Null}((\mathbb K-\mu)^{\alpha}),\quad
Q(\lambda)=Q_{\lambda,\mu}={\rm Null}(\mathbb K-\mu)
\end{eqnarray*}
denote the algebraic and geometric eigenspaces,
 respectively.
The subspaces $M(\lambda)$ and $Q(\lambda)\subset M(\lambda)$ are finite dimensional. The numbers
$m={\rm dim}M(\lambda)$ and $q={\rm dim}Q(\lambda)$ are called the algebraic
and the geometric multiplicities
of $\mu$ (and $\lambda$). The vectors in $M(\lambda)$ are generalized eigenvectors.
 The order of a generalized
eigenvector is the smallest integer $p$ such that $(\mathbb K-\mu)^p\mathbf U=0$
(vectors in $Q(\lambda)$ being generalized eigenvectors of
order $1$). Let us point out that a generalized eigenvector $\mathbf U^p$ of
order $p$ satisfies
\begin{eqnarray}\label{Recurisve_Law}
a(\mathbf U^p,\mathbf \Psi)&=&\lambda b(\mathbf U^p,\mathbf \Psi)
+\lambda a(\mathbf U^{p-1},\mathbf \Psi),\ \ \ \ \forall \mathbf\Psi\in \mathbb V,
\end{eqnarray}
where $\mathbf U^{p-1}$ is a generalized eigenvector of order $p-1$.

Similarly we define the spaces of (generalized) eigenvectors for the adjoint problem
\begin{eqnarray*}
M^*(\lambda)=M_{\lambda,\mu}^*={\rm Null}((\mathbb K_*-\bar{\mu})^{\alpha}),\quad Q^*(\lambda)
=Q^*_{\lambda,\mu}={\rm Null}(\mathbb K_*-\bar{\mu}).
\end{eqnarray*}
Note that $\mu$ is an eigenvalue of $\mathbb K$ ($\lambda$ is an eigenvalue of problem
(\ref{Problem_Weak})) if and only if
$\bar{\mu}$ is an eigenvalue of $\mathbb K_*$ (${ \bar{\lambda}}$ is an eigenvalue of adjoint problem
(\ref{Eigenvalue_Problem_Weak_Adjoint})) with the ascent $\alpha$
and the  algebraic multiplicity $m$ for both eigenvalues being the same.

\section{Finite element method for Transmission eigenvalue problem}\label{FEM_Method}
Now, let us define the finite element approximations for the problem
(\ref{Problem_Weak}). First we generate a shape-regular
decomposition of the computational domain $\Omega\subset \mathcal{R}^d\
(d=2,3)$ into triangles or rectangles for $d=2$ (tetrahedrons or
hexahedrons for $d=3$). The diameter of a cell $K\in\mathcal{T}_h$
is denoted by $h_K$. The mesh diameter $h$ describes the maximum
diameter of all cells $K\in\mathcal{T}_h$. Based on the mesh
$\mathcal{T}_h$, { we construct} a finite element space denoted by
$\mathbb V_h\subset \mathbb V$. The same argument as in the beginning of this section
illustrates that the following discrete inf-sup conditions also hold
\begin{eqnarray}
\|\mathbf \Phi_h\|_{\mathbb V} \lesssim \sup_{\mathbf \Psi_h\in \mathbb V_h}
\frac{a(\mathbf \Phi_h,\mathbf\Psi_h)}{\|\mathbf \Psi_h\|_{\mathbb V}}\ \ {\rm and}
\ \ \|\mathbf \Phi_h\|_{\mathbb V} \lesssim \sup_{\mathbf \Psi_h\in \mathbb V_h}
\frac{a(\mathbf\Psi_h,\mathbf \Phi_h)}{\|\mathbf \Psi_h\|_{\mathbb V}}.
\end{eqnarray}

The standard finite element method for the problem (\ref{Problem_Weak}) is defined as follows:
 Find $(\lambda_h,\mathbf U_h)\in \mathcal{C}\times \mathbb V_h$ such that $\|\mathbf U_h\|_{\mathbb W}=1$ and
\begin{eqnarray}\label{Eigenvalue_Problem_Weak_Discrete}
a(\mathbf U_h,\mathbf \Psi_h)&=&\lambda_h b(\mathbf U_h,\mathbf\Psi_h),\ \ \ \ \ \forall \mathbf\Psi_h\in \mathbb V_h.
\end{eqnarray}
Similarly, the discretization for the adjoint problem (\ref{Eigenvalue_Problem_Weak_Adjoint})
can be defined as:
Find $(\lambda_h,\mathbf U_h^*)\in \mathcal{C}\times \mathbb V_h$ such that $\|\mathbf U_h^*\|_{\mathbb W}=1$ and
\begin{eqnarray}\label{Eigenvalue_Problem_Weak_Discrete_Adjoint}
a(\mathbf\Psi_h,\mathbf U_h^*)&=&\lambda_h b(\mathbf\Psi_h,\mathbf U_h^*),\ \ \ \ \ \forall \mathbf\Psi_h\in \mathbb V_h.
\end{eqnarray}

By introducing Galerkin projections $P_h,\ P_h^*\in \mathcal{L}(\mathbb V,\mathbb V_h)$
with the following equations
\begin{eqnarray*}
a(P_h\mathbf\mathbf\Phi,\mathbf \Psi_h)&=&a(\mathbf\mathbf\Phi,\mathbf \Psi_h),\ \ \ \ \forall \mathbf\Psi_h\in\mathbb V_h,\\
a(\mathbf \Psi_h, P_h^*\mathbf\mathbf\Phi)&=&a(\mathbf \Psi_h, \mathbf\mathbf\Phi),
\ \ \ \ \forall \mathbf\Psi_h\in\mathbb V_h,
\end{eqnarray*}
the equation (\ref{Eigenvalue_Problem_Weak_Discrete}) can be rewritten as an operator
 form with $\mu_h:=\lambda_h^{-1}$ ($P_h$ is a bounded operator),
\begin{eqnarray}
P_h\mathbb K\mathbf U_h=\mu_h\mathbf U_h.
\end{eqnarray}
Similarly for the adjoint problem (\ref{Eigenvalue_Problem_Weak_Discrete_Adjoint}), we have
\begin{eqnarray}
P_h^*\mathbb K_*\mathbf U_h^*=\bar{\mu}_h\mathbf U_h^*.
\end{eqnarray}

Let $\mu$ be an eigenvalue (with algebraic multiplicity $m$)
of the compact operator $\mathbb K$.
If $\mathbb K$ is approximated by a sequence of compact operators $\mathbb K_h$
converging to $\mathbb K$ in norm, i.e.,
$\lim\limits_{h\rightarrow 0+}{ \|\mathbb K-\mathbb K_h\|_{\mathbb V}}=0$,
then for $h$ sufficiently small $\mu$ is approximated by
exactly $m$ eigenvalues $\{\mu_{j,h}\}_{j=1,\cdots,m}$
(counted according to their algebraic multiplicities)
of $\mathbb K_h$, i.e.,
\begin{eqnarray*}
\lim_{h\rightarrow 0+}\mu_{j,h}=\mu\ \ \ \ \ {\rm for}\ j=1,\cdots,m.
\end{eqnarray*}
The space of generalized eigenvectors of $\mathbb K$ is approximated by the subspace
\begin{eqnarray}
M_h(\lambda)=M_h^{\lambda,\mu}=\sum_{j=1}^m{\rm Null}((\mathbb K_h-\mu_{j,h})^{\alpha_{\mu_{j,h}}}),
\end{eqnarray}
where $\alpha_{\mu_{j,h}}$ is the smallest integer such that
${\rm Null}\big((\mathbb K_h-\mu_{j,h})^{\alpha_{\mu_{j,h}}}\big)
={\rm Null}\big((\mathbb K_h-\mu_{j,h})^{\alpha_{\mu_{j,h}}+1}\big)$.
We similarly define the space $Q_h(\lambda)=Q_h^{\lambda,\mu}=\sum_{j=1}^m{\rm Null}(\mathbb K_h-\mu_{j,h})$ and the
counterparts $M_h^*(\lambda)$, $Q_h^*(\lambda)$ for
the adjoint problem.

Now, we describe a computational scheme to produce the algebraic eigenspace $M_h(\lambda)$ from the geometric eigenspace $Q_h(\lambda)=\{\mathbf U_{1,h},\cdots,\mathbf U_{q,h}\}$
corresponding to eigenvalues $\{\lambda_{1,h},\cdots,\lambda_{q,h}\}$, which
 converge to the same eigenvalue $\lambda$.

Starting from all eigenfunctions in the geometric eigenspace $Q_h(\lambda)$ (of order $1$),
we use the following recursive process to compute
 algebraic eigenspaces (c.f. \cite{Shaidurov})
\begin{equation}
\hskip-0.2cm \left\{
\begin{array}{rcl}
a(\mathbf U_{j,h}^{p},\mathbf \Psi_h)-\lambda_{j,h}
b(\mathbf U_{j,h}^{p},\mathbf \Psi_h)&=&
\lambda_{j,h}a(\mathbf U_{j,h}^{p-1},\mathbf\Psi_h),
\ \forall \mathbf \Psi_h\in \mathbb  V_h,\\
b(\mathbf U_{j,h}^{p}, \mathbf\Psi_h)&=&0,\ \ \ \
\quad\quad\quad\quad\quad\ \forall \mathbf \Psi_h\in Q_h({\lambda}),
\end{array}
\right.
\end{equation}
where $p\geq 2$, $\mathbf U_{j,h}^{p}$ is the general eigenfunction of order
$p$ and $\mathbf U_{j,h}^1=\mathbf U_{j,h}\in Q_h({\lambda})$ for $j=1,\cdots,q$.

With the above process, we generate the algebraic eigenspace
$$M_h(\lambda)=\{\mathbf U_{1,h},\cdots, \mathbf U_{q,h},\cdots, \mathbf U_{m,h}\}$$
corresponding to eigenvalues $\{\lambda_{1,h},\cdots,\lambda_{q,h},\cdots, \lambda_{m,h}\}$,
 which converge to the same eigenvalue $\lambda$.  Similarly, we can produce the
adjoint algebraic eigenspace $M_h^*(\lambda)$ from the geometric eigenspace $Q_h^*(\lambda)$.

For two linear spaces $A$ and $B$, we denote
\begin{eqnarray*}
\widehat{\Theta}(A,B) = \sup_{\mathbf \Phi\in A,\|\mathbf \Phi\|_{\mathbb V}=1}
\inf_{\mathbf\Psi\in B}\|\mathbf\Phi-\mathbf\Psi\|_{\mathbb V},\ \
\widehat{\Phi}(A,B) = \sup_{\mathbf \Phi\in A,\|\mathbf \Phi\|_{\mathbb W}=1}
\inf_{\mathbf\Psi\in B}\|\mathbf\Phi-\mathbf\Psi\|_{\mathbb W},
\end{eqnarray*}
 and define gaps between $A$ and $B$ in $\|\cdot\|_{\mathbb V}$ as
\begin{eqnarray}
\Theta(A,B)=\max\big\{\widehat{\Theta}(A,B), \widehat{\Theta}(B,A)\big\},
\end{eqnarray}
and in $\|\cdot\|_{\mathbb W}$ as
\begin{eqnarray}
\Phi(A,B)=\max\big\{\widehat{\Phi}(A,B), \widehat{\Phi}(B,A)\big\}.
\end{eqnarray}

Before introducing the convergence results of the finite element approximation for
nonsymmetric eigenvalue problems, we define the following notations
\begin{eqnarray}
&&\delta_h(\lambda)=\sup_{\mathbf U\in M(\lambda),\|\mathbf U\|_{\mathbb V}=1}
\inf_{\mathbf \Psi_h\in \mathbb V_h}\|\mathbf U-\mathbf \Psi_h\|_{\mathbb V},\\
&&\delta_h^*(\lambda)=\sup_{\mathbf U\in M^*(\lambda),\|\mathbf U\|_{\mathbb V}=1}
\inf_{\mathbf \Psi_h\in \mathbb V_h}\|\mathbf U-\mathbf\Psi_h\|_{\mathbb V},\\
&&\rho_h(\lambda)=\sup_{\mathbf U\in M(\lambda),\|\mathbf U\|_{\mathbb W}=1}
\inf_{\mathbf\Psi_h\in \mathbb V_h}\|\mathbf U-\mathbf\Psi_h\|_{\mathbb W},\\
&&\rho_h^*(\lambda)=\sup_{\mathbf U\in M^*(\lambda),\|\mathbf U\|_{\mathbb W}=1}
\inf_{\mathbf\Psi_h\in \mathbb  V_h}\|\mathbf U-\mathbf\Psi_h\|_{\mathbb W},\\
&&\eta_a(h)=\sup_{\mathbf\Phi\in\mathbb V,\|\mathbf\Phi\|_{\mathbb W}=1}
\inf_{\mathbf\Psi_h\in \mathbb V_h}\|\mathbb K \mathbf\Phi-\mathbf\Psi_h\|_{\mathbb V},\\
&&\eta_a^*(h)=\sup_{\mathbf\Phi\in\mathbb V,\|\mathbf\Phi\|_{\mathbb W}=1}
\inf_{\mathbf\Psi_h\in \mathbb V_h}\|\mathbb K_* \mathbf\Phi-\mathbf\Psi_h\|_{\mathbb V}.
\end{eqnarray}
In order to derive error bounds for eigenpair approximations in the
weak norm $\|\cdot\|_{\mathbb W}$, we need the following error estimates in the weak norm
$\|\cdot\|_{\mathbb W}$
of the finite element approximation.
\begin{lemma}\label{Negative_norm_estimate_Lemma}
(\cite[Lemma 3.3 and Lemma 3.4]{BabuskaOsborn})
\begin{eqnarray}
\lim_{h\rightarrow0}\eta_a(h)=0,\ \ \ \lim_{h\rightarrow0}\eta_a^*(h)=0,
\end{eqnarray}
and
\begin{eqnarray}\label{Negative_norm_Error}
\rho_h(\lambda)&\lesssim& \eta_a^*(h)\delta_h(\lambda),\\
\rho_h^*(\lambda)&\lesssim& \eta_a(h)\delta_h^*(\lambda).
\end{eqnarray}
\end{lemma}
Base on the general theory of the error estimates for the eigenvalue problems by the finite element method
\cite[Section 8]{BabuskaOsborn}, we have the following results for the transmission eigenvalue problem.
\begin{theorem}\label{Error_Estimate_Theorem}
When the mesh size $h$ is small enough, we have
\begin{eqnarray}
&&\Theta(M(\lambda),M_h(\lambda))\lesssim \delta_h(\lambda),\ \ \
\Theta(M^*(\lambda),M^*_h(\lambda))\lesssim \delta^*_h(\lambda),\\
&&\Phi(M(\lambda),M_h(\lambda))\lesssim \rho_h(\lambda),\ \ \
\Phi(M^*(\lambda),M^*_h(\lambda))\lesssim \rho^*_h(\lambda),\\
&&|\lambda-\widehat{\lambda}_{h}|\lesssim \delta_h(\lambda)\delta_h^*(\lambda),
\end{eqnarray}
where $\widehat{\lambda}_h={ \frac{1}{m}}\sum_{j=1}^m\lambda_{j,h}$ with $\lambda_{1,h},\cdots,\lambda_{m,h}$
converging to $\lambda$.
\end{theorem}
\section{Multilevel correction method for transmission eigenvalue problem}\label{Multilevel_Correction_Method}
In this section, we introduce a type of multilevel correction method for the
transmission eigenvalue problem. This multilevel correction method
consists of solving some auxiliary linear problems
in a sequence of finite element spaces and an eigenvalue problem in a very low dimensional space.
For more discussion about the multilevel correction method,
please refer to \cite{LinXie_Steklov,LinXie,Xie_Nonconforming}.

In order to do multilevel correction scheme, we first generate a coarse mesh $\mathcal{T}_H$
with the mesh size $H$ and the coarse linear finite element space $\mathbb V_H$ is
defined on the mesh $\mathcal{T}_H$ \cite{BrennerScott,Ciarlet}. Then we define a sequence of triangulations
 $\{\mathcal{T}_{h_{\ell}}\}_{\ell}$
of $\Omega\subset \mathcal{R}^d$ determined as follows.
Suppose $\mathcal{T}_{h_1}$ (produced from $\mathcal{T}_H$ by
regular refinements) is given and let $\mathcal{T}_{h_{\ell}}$ be obtained
from $\mathcal{T}_{h_{\ell-1}}$ via regular refinement (produce $\beta^d$ subelements) such that
\begin{eqnarray}\label{Mesh_Size_Relation}
h_{\ell}\approx\frac{1}{\beta}h_{\ell-1},
\end{eqnarray}
where the integer $\beta>1$ denotes the refinement index \cite{BrennerScott,Shaidurov}.
It always equals $2$ in the first three numerical experiments
with quasi-uniform refinement.
Based on this sequence of meshes, we construct the corresponding linear finite element spaces such that
\begin{eqnarray}\label{FEM_Space_Series}
\mathbb V_{H}\subseteq \mathbb V_{h_1}\subset \mathbb V_{h_2}\subset\cdots\subset \mathbb V_{h_n}.
\end{eqnarray}

Before designing the multilevel correction method, we first introduce a type of
one correction step which can improve the accuracy of the given eigenpair approximation
by solving a linear problem and an eigenvalue problem in a very low dimensional space.
Assume that we have obtained the algebraic eigenpair approximations
$(\lambda_{j}^{h_{\ell}},\mathbf U_{j}^{h_{\ell}})\in\mathcal{C}\times \mathbb V_{h_{\ell}}$ and
the corresponding adjoint ones
$(\bar{\lambda}_{j}^{h_{\ell}},\mathbf U_{j}^{h_{\ell},*})\in\mathcal{C}\times \mathbb V_{h_{\ell}}$ for $j=i,\cdots,i+m-1$, where
eigenvalues $\{\lambda_{j}^{h_{\ell}}\}_{j=i}^{i+m-1}$ converge to the desired eigenvalue
$\lambda_i$ of (\ref{Problem_Weak}) with multiplicity $m$.
Now we introduce a correction step to improve the accuracy of the
current eigenpair approximations.
Let $\mathbb V_{h_{\ell+1}}\subset V$ be the conforming finite element space based on a finer
mesh $\mathcal{T}_{h_{\ell+1}}$ which is produced by refining $\mathcal{T}_{h_{\ell}}$.
We start from a conforming linear finite element space $\mathbb V_H$ on the
coarsest mesh $\mathcal{T}_H$ to design the following one correction step.

\begin{algorithm}\label{Correction_Step}
One Correction Step

\begin{enumerate}
\item  For $j=i,\cdots,i+m-1$ Do

Solve the following two boundary value problems:

\vspace{0.1cm}
$-$ Find $\widetilde{\mathbf U}_{j,h_{\ell+1}}\in\mathbb V_{h_{\ell+1}}$ such that
\begin{eqnarray}\label{aux_problem}
a(\widetilde{\mathbf U}_{j,h_{\ell+1}},\mathbf \Psi_{h_{\ell+1}})&=&b(\mathbf U_{j}^{h_{\ell}},\mathbf\Psi_{h_{\ell+1}}),\ \
\ \forall \mathbf \Psi_{h_{\ell+1}}\in \mathbb V_{h_{\ell+1}}.
\end{eqnarray}
$-$ Find  $\widetilde{\mathbf U}^*_{j,h_{\ell+1}}\in \mathbb V_{h_{\ell+1}}$ such that
\begin{eqnarray}\label{aux_problem_Adjoint}
a(\mathbf\Psi_{h_{\ell+1}},\widetilde{\mathbf U}^*_{j,h_{\ell+1}})&=&b(\mathbf \Psi_{h_{\ell+1}}, \mathbf U^{h_{\ell},*}_{j}),\ \
\ \forall \mathbf \Psi_{h_{\ell+1}}\in \mathbb V_{h_{\ell+1}}.
\end{eqnarray}
End Do
\item  Define two new finite element { spaces}
\begin{eqnarray*}
\mathbb V_{H,h_{\ell+1}}=\mathbb V_H\oplus{\rm span}
\{\widetilde{\mathbf U}_{i,h_{\ell+1}},\cdots,\widetilde{\mathbf U}_{i+m-1,h_{\ell+1}}\}
\end{eqnarray*}
and
\begin{eqnarray*}
\mathbb V^*_{H,h_{\ell+1}}=\mathbb V_H\oplus{\rm span}
\{\widetilde{\mathbf U}^*_{i,h_{\ell+1}},\cdots,\widetilde{\mathbf U}^*_{i+m-1,h_{\ell+1}}\}.
\end{eqnarray*}

Solve the following two eigenvalue problems:

\vspace{0.2cm}
$-$ Find $(\lambda_{j}^{h_{\ell+1}},\mathbf U_{j}^{h_{\ell+1}})\in\mathcal{C}\times \mathbb V_{H,h_{\ell+1}}$ such
that $\|\mathbf U_{j}^{h_{\ell+1}}\|_{\mathbb W}=1$ and
\begin{eqnarray}\label{Eigen_Augment_Problem}
\hskip-1cm a(\mathbf U_{j}^{h_{\ell+1}},\mathbf\Psi_{H,h_{\ell+1}})=\lambda_{j}^{h_{\ell+1}}
 b(\mathbf U_{j}^{h_{\ell+1}},\mathbf\Psi_{H,h_{\ell+1}}),\ \
\forall \mathbf\Psi_{H,h_{\ell+1}}\in \mathbb V^*_{H,h_{\ell+1}}.
\end{eqnarray}
$-$ Find $({\lambda}_{j}^{h_{\ell+1}},\mathbf U_{j}^{h_{\ell+1},*})\in\mathcal{C}\times \mathbb V^*_{H,h_{\ell+1}}$ such
that $\|\mathbf U_{j}^{h_{\ell+1},*}\|_{\mathbb W}=1$ and
\begin{eqnarray}\label{Eigen_Augment_Problem_Adjoint}
\hskip-1cm a(\mathbf\Psi_{H,h_{\ell+1}},\mathbf U_{j}^{h_{\ell+1},*})={\lambda}_{j}^{h_{\ell+1}}
b(\mathbf\Psi_{H,h_{\ell+1}},\mathbf U_{j}^{h_{\ell+1},*}),\ \
\forall \mathbf\Psi_{H,h_{\ell+1}}\in \mathbb V_{H,h_{\ell+1}}.
\end{eqnarray}

\item Choose $2q$ eigenpairs $\{\lambda_{j}^{h_{\ell+1}}, \mathbf U_{j}^{h_{\ell+1}}\}_{j=i}^{i+q-1}$ and
$\{\lambda_{j}^{h_{\ell+1}}, \mathbf U_{j}^{h_{\ell+1},*}\}_{j=i}^{i+q-1}$ to define  
 two new geometric eigenspaces
\begin{eqnarray*}
Q_{h_{\ell+1}}(\lambda_i)={\rm span}\big\{\mathbf U_{i}^{h_{\ell+1}},\cdots, \mathbf U_{i+q-1}^{h_{\ell+1}}\big\}
\end{eqnarray*}
and
\begin{eqnarray*}
Q^*_{h_{\ell+1}}(\lambda_i)={\rm span}\big\{\mathbf U_{i}^{h_{\ell+1},*},\cdots, \mathbf U_{i+q-1}^{h_{\ell+1},*}\big\}.
\end{eqnarray*}
Based on these two geometric eigenspaces $Q_{h_{\ell+1}}(\lambda_i)$ and  $Q^*_{h_{\ell+1}}(\lambda_i)$,
compute two algebraic eigenspaces
\begin{eqnarray}
M_{h_{\ell+1}}(\lambda_i)={\rm span}\big\{\mathbf U_{i}^{h_{\ell+1}},\cdots, \mathbf U_{i+m-1}^{h_{\ell+1}}\big\}
\end{eqnarray}
and
\begin{eqnarray}
M^*_{h_{\ell+1}}(\lambda_i)={\rm span}\big\{\mathbf U_{i}^{h_{\ell+1},*},\cdots, \mathbf U_{i+m-1}^{h_{\ell+1},*}\big\}.
\end{eqnarray}
\end{enumerate}
In order to simplify the notations and summarize the above three steps, we define
\begin{eqnarray*}
&&\big(\{\lambda_{j}^{h_{\ell+1}}\}_{j=i}^{i+m-1},M_{h_{\ell+1}}(\lambda_i),M^*_{h_{\ell+1}}(\lambda_i)\big)=\nonumber\\
&&\ \ \ \ \quad\quad {\it
Correction}\big(\mathbb V_H,\{\lambda_{j}^{h_{\ell}}\}_{j=i}^{i+m-1},M_{h_{\ell}}(\lambda_i),M^*_{h_{\ell}}(\lambda_i),\mathbb V_{h_{\ell+1}}\big).
\end{eqnarray*}
\end{algorithm}
\begin{remark}
Since in Step 1 of Algorithm \ref{Correction_Step}, the solving processes for the boundary value problems
are independent of each other for different $j$, we can implement them in parallel.
\end{remark}

\begin{theorem}\label{Error_Estimate_One_Correction_Theorem}
Assume the given eigenpairs $\big(\{\lambda_{j}^{h_{\ell}}\}_{j=i}^{i+m-1}, M_{h_{\ell}}(\lambda_i),M^*_{h_{\ell}}(\lambda_i)\big)$
in  Algorithm \ref{Correction_Step} have the following error estimates
\begin{eqnarray}
\Theta(M(\lambda_i),M_{h_{\ell}}(\lambda_i)) &\lesssim& \varepsilon_{h_{\ell}}(\lambda_i),\\
\Theta(M^*(\lambda_i),M^*_{h_{\ell}}(\lambda_i)) &\lesssim& \varepsilon^*_{h_{\ell}}(\lambda_i),\label{Error_u_h_1}\\
\Phi(M(\lambda_i),M_{h_{\ell}}(\lambda_i)) &\lesssim&\eta_a^*(H)\varepsilon_{h_{\ell}}(\lambda_i),\\
\Phi(M^*(\lambda_i),M^*_{h_{\ell}}(\lambda_i))&\lesssim& \eta_a(H)\varepsilon^*_{h_{\ell}}(\lambda_i). \label{Error_u_h_1_nagative}
\end{eqnarray}
Then after one correction step, the resultant eigenpair approximations\\
$(\{\lambda_{j}^{h_{\ell+1}}\}_{j=i}^{i+m-1},M_{h_{\ell+1}}(\lambda_i),M^*_{h_{\ell+1}}(\lambda_i))$
have the following error estimates
\begin{eqnarray}
\Theta(M(\lambda_i),M_{h_{\ell+1}}(\lambda_i)) &\lesssim& \varepsilon_{h_{\ell+1}}(\lambda_i),\label{Estimate_u_u_h_2}\\
\Theta(M^*(\lambda_i),M^*_{h_{\ell+1}}(\lambda_i)) &\lesssim& \varepsilon^*_{h_{\ell+1}}(\lambda_i),\label{Estimate_u_u_h_2_adjoint}\\
\Phi(M(\lambda_i),M_{h_{\ell+1}}(\lambda_i)) &\lesssim&
\eta_a^*(H) \varepsilon_{h_{\ell+1}}(\lambda_i),\label{Estimate_u_h_2_Nagative}\\
\Phi(M^*(\lambda_i),M^*_{h_{\ell+1}}(\lambda_i)) &\lesssim&
\eta_a(H) \varepsilon^*_{h_{\ell+1}}(\lambda_i),\label{Estimate_u_h_2_Nagative_Adjoint}
\end{eqnarray}
where
\begin{eqnarray*}
\varepsilon_{h_{\ell+1}}(\lambda_i)&:=&\eta_a^*(H)\varepsilon_{h_{\ell}}(\lambda_i)+\delta_{h_{\ell+1}}(\lambda_i), \\
\varepsilon^*_{h_{\ell+1}}(\lambda_i)&:=&\eta_a(H)\varepsilon^*_{h_{\ell}}(\lambda_i)+\delta^*_{h_{\ell+1}}(\lambda_i).
\end{eqnarray*}
\end{theorem}
\begin{proof}
From (\ref{Recurisve_Law}),  there exist the basis functions $\big\{\mathbf U_j\big\}_{j=i}^{i+m-1}$ of $M(\lambda_i)$ 
such that 
\begin{eqnarray}
a(\mathbf U_j,\mathbf \Psi)&=&b\left(\sum_{k=i}^{i+m-1}p_{jk}(\lambda_i)\mathbf U_k,\mathbf \Psi\right),\ \ \ \forall \mathbf \Psi\in \mathbb V,
\end{eqnarray}
where $p_{jk}$ denotes a polynomial of degree no more than $\alpha$ for $k=i,\cdots,j$ with $p_{jj}(\lambda_i)=\lambda_i$ 
and $p_{jk}(\lambda_i)=0$ for $j<k\leq i+m-1$. 
We can define a matrix $\mathcal P:=(p_{j+1-i,k+1-i})_{i\leq j,k\leq i+m-1}\in \mathcal C^{m\times }$ such that    
\begin{eqnarray}
a(\bar{\mathbf U},\mathbf \Psi)=b(\mathcal P\bar{\mathbf U},\mathbf \Psi),\ \ \ \ \forall \mathbf\Psi\in \mathbb V,
\end{eqnarray}
where $\bar{\mathbf U}:=(\mathbf U_i,\cdots,\mathbf U_{i+m-1})^T$. 
It is easy to know  that the matrix $\mathcal P$ is nonsingular providing $\lambda_i\neq 0$.


For each $\widetilde{\mathbf U}_{j,h_{\ell+1}}$, from the definitions of $\Theta(M(\lambda_i),M_{h_{\ell}}(\lambda_i))$ and $\Phi(M(\lambda_i),M_{h_{\ell}}(\lambda_i))$,
there exist a vector $\mathcal R_j := (c_1,\cdots,c_m)^T\in \mathcal C^{m\times 1}$ such that
\begin{eqnarray}
\|\mathbf U_{j}^{h_{\ell}}-\mathcal R_j^T\bar{\mathbf U}\|_{\mathbb V} &\lesssim & \varepsilon_{h_{\ell}}(\lambda_i),
\ \ \quad \ \ \ \ \ \ \ {\rm for}\ j=i,\cdots,i+m-1,\\
\|\mathbf U_{j}^{h_{\ell}}-\mathcal R_j^T\bar{\mathbf U}\|_{\mathbb W} &\lesssim& \eta_a^*(H)\varepsilon_{h_{\ell}}(\lambda_i),
\ \ \ \ {\rm for}\ j=i,\cdots,i+m-1.
\end{eqnarray}
For any $\mathbf\Psi_{h_{\ell+1}}\in \mathbb V_{h_{\ell+1}}$, we have
\begin{eqnarray}\label{Error_Estimate_1}
&&{|}a(\widetilde{\mathbf U}_{j,h_{\ell+1}}-P_{h_{\ell+1}}\mathcal R_j\mathcal P^{-1}\bar{\mathbf U},\mathbf\Psi_{h_{\ell+1}}){|}
={|}a(\widetilde{\mathbf U}_{j,h_{\ell+1}}-\mathcal R_j\mathcal P^{-1}\bar{\mathbf U},\mathbf\Psi_{h_{\ell+1}}){|}\nonumber\\
&=&b(\mathbf U_j^{h_{\ell}}-\mathcal R_j^T\mathcal{P}^{-1}\mathcal P\bar{\mathbf U},\mathbf \Psi_{h_{\ell+1}})
= {|}b(\mathbf U_{j}^{h_{\ell}}-\mathcal R_j^T\bar{\mathbf U},\mathbf\Psi_{h_{\ell+1}}){|}\nonumber\\
&\lesssim& \eta_a^*(H)\varepsilon_{h_{\ell}}(\lambda_i)\|\mathbf\Psi_{h_{\ell+1}}\|_{\mathbb V}.
\end{eqnarray}
From (\ref{Inf-sup_left}) and  (\ref{Error_Estimate_1}), the following estimate holds
\begin{eqnarray}
\|\widetilde{\mathbf U}_{j,h_{\ell+1}}-P_{h_{\ell+1}}\mathcal R_j^T\mathcal P^{-1}\bar{\mathbf U}\|_{\mathbb V} &\lesssim&
 \eta_a^*(H)\varepsilon_{h_{\ell}}(\lambda_i),\nonumber\\
 && \ \ {\rm for}\ j=i,\cdots,i+m-1.
\end{eqnarray} 
Combining with the error estimate
\begin{eqnarray}
\|\mathcal R_j^T\mathcal P^{-1}\bar{\mathbf U}-P_{h_{\ell+1}}\mathcal R_j^T\mathcal P^{-1}\bar{\mathbf U}\|_{\mathbb V}
&\lesssim& \delta_{h_{\ell+1}}(\lambda_i),\nonumber\\
&&\ \ {\rm for}\ j=i,\cdots,i+m-1,
\end{eqnarray}
we have
\begin{eqnarray}\label{Error_Estimates_Tilde_u_J}
\|\widetilde{\mathbf U}_{j,h_{\ell+1}}-\mathcal R_j^T\mathcal P^{-1}\bar{\mathbf U}\|_{\mathbb V} &\lesssim&
\eta_a^*(H)\varepsilon_{h_{\ell}}(\lambda_i)+\delta_{h_{\ell+1}}(\lambda_i),\nonumber\\
&&\ \ \ \ \  \ \ {\rm for}\ j=i,\cdots,i+m-1.
\end{eqnarray}
After Step 3, from the definition of $\mathbb V_{H,h_{\ell+1}}$ and (\ref{Error_Estimates_Tilde_u_J}), we derive
\begin{eqnarray}\label{Definition_Varepsilon_k+1}
&&\sup_{\mathbf U\in M(\lambda_i),\|\mathbf U\|_{\mathbb V}=1}
\inf_{\mathbf\Psi_{H,h_{\ell+1}}\in \mathbb V_{H,h_{\ell+1}}}\|\mathbf U-\mathbf\Psi_{H,h_{\ell+1}}\|_{\mathbb V}\nonumber\\
&\leq& \sup_{\mathbf U\in M(\lambda_i),\|\mathbf U\|_{\mathbb V}=1}
\inf_{\mathbf\Psi_{h_{\ell+1}}\in \mathbb V_{h_{\ell+1}}}\|\mathbf U-\mathbf\Psi_{h_{\ell+1}}\|_{\mathbb V}\nonumber\\
&\lesssim& \sup_{\mathbf\Psi_{h_{\ell+1}}\in \mathbb V_{h_{\ell+1}}, \|\mathbf\Psi_{h_{\ell+1}}\|_{\mathbb V}=1}
\inf_{\mathbf U\in M(\lambda_i)}\|\mathbf\Psi_{h_{\ell+1}}-\mathbf U\|_{\mathbb V}\nonumber\\
&\lesssim& \max_{j=i,\cdots,i+m-1}\|\widetilde{\mathbf U}_{j,h_{\ell+1}}-\mathcal R_j^T\mathcal P^{-1}\bar{\mathbf U}\|_{\mathbb V}\nonumber\\
&\lesssim& \eta_a^*(H)\varepsilon_{h_{\ell}}(\lambda_i)+\delta_{h_{\ell+1}}(\lambda_i),
\end{eqnarray}
where $ \mathbb V_{h_{\ell+1}}:={\rm span}\{\widetilde{\mathbf U}_i^{h_{\ell+1}},\cdots,\widetilde{\mathbf U}_{i+m-1}^{h_{\ell+1}}\}$. 

Similarly, we have 
\begin{eqnarray}\label{Definition_Varepsilon_k+1_Adjoint}
&&\sup_{\mathbf U_*\in M^*(\lambda_i),\|\mathbf U_*\|_{\mathbb V}=1}\inf_{\mathbf \Psi_{H,h_{\ell+1}}\in
 \mathbb V^*_{H,h_{\ell+1}}}\|\mathbf U^*-\mathbf \Psi_{H,h_{\ell+1}}\|_{\mathbb V}\nonumber\\
 &\lesssim& \eta_a(H)\varepsilon^*_{h_{\ell}}(\lambda_i)+\delta^*_{h_{\ell+1}}(\lambda_i).
\end{eqnarray}
Then from the error estimate results stated in Theorem \ref{Error_Estimate_Theorem}
for the eigenvalue problem (see, e.g., \cite[Section 8]{BabuskaOsborn})
and (\ref{Definition_Varepsilon_k+1})-(\ref{Definition_Varepsilon_k+1_Adjoint}),
the following error estimates hold
\begin{eqnarray}
\Theta(M(\lambda_i),M_{h_{\ell+1}}(\lambda_i))
&\lesssim& \eta_a^*(H)\varepsilon_{h_{\ell}}(\lambda_i)+\delta_{h_{\ell+1}}(\lambda_i),\\
\Theta(M^*(\lambda_i),M^*_{h_{\ell+1}}(\lambda_i))
&\lesssim& \eta_a(H)\varepsilon^*_{h_{\ell}}(\lambda_i)+\delta^*_{h_{\ell+1}}(\lambda_i).
\end{eqnarray}
These are the desired estimates (\ref{Estimate_u_u_h_2}) and (\ref{Estimate_u_u_h_2_adjoint}).
Furthermore,
\begin{eqnarray}
&&\Phi(M(\lambda_i),M_{h_{\ell+1}}(\lambda_i))\nonumber\\
&\lesssim &
\widetilde{\eta}_a^*(H)\sup_{\mathbf U\in M(\lambda_i),\|\mathbf U\|_{\mathbb V}=1}
\inf_{\mathbf\Psi_{H,h_{\ell+1}}\in \mathbb V_{H,h_{\ell+1}}}\|\mathbf U-\mathbf\Psi_{H,h_{\ell+1}}\|_{\mathbb V}\nonumber\\
&\leq & \eta_a^*(H)\varepsilon_{h_{\ell+1}}(\lambda_i),
\end{eqnarray}
where
\begin{eqnarray}
\widetilde{\eta}_a^*(H):=\sup_{\mathbf f\in \mathbb V,\|\mathbf f\|_{\mathbb W}=1}
\inf_{\mathbf\Psi_{H,h_{\ell+1}}\in \mathbb V_{H,h_{\ell+1}}}
\|\mathbb K_*\mathbf f-\mathbf\Psi_{H,h_{\ell+1}}\|_{\mathbb V}\leq \eta_a^*(H).
\end{eqnarray}
Then we obtain (\ref{Estimate_u_h_2_Nagative}). A similar argument leads to (\ref{Estimate_u_h_2_Nagative_Adjoint}).
\end{proof}


Now, based on the {\it One Correction Step} defined in Algorithm \ref{Correction_Step},
we introduce a multilevel correction scheme for the transmission eigenvalue problem.

\begin{algorithm}\label{Multi_Correction}
Multilevel Correction Scheme
\begin{enumerate}
\item Construct a coarse conforming finite element space $\mathbb V_{h_1}$ on $\mathcal{T}_{h_1}$
such that $\mathbb V_H\subset \mathbb V_{h_1}$
and solve the following two eigenvalue problems:

\vspace{0.1cm}
$-$ Find $(\lambda^{h_1},\mathbf U^{h_1})\in \mathcal{C}\times \mathbb V_{h_1}$ such that
$\|\mathbf U^{h_1}\|_{\mathbb W}=1$ and
\begin{eqnarray}\label{Initial_Eigen_Problem}
a(\mathbf U^{h_1},\mathbf\Psi_{h_1})&=&\lambda^{h_1}b(\mathbf U^{h_1},\mathbf\Psi_{h_1}),\ \ \ \
\forall \mathbf\Psi_{h_1}\in \mathbb V_{h_1}.
\end{eqnarray}
$-$ Find $(\lambda^{h_1},\mathbf U^{h_1,*})\in \mathcal{C}\times \mathbb V_{h_1}$ such that
$\|\mathbf U^{h_1,*}\|_{\mathbb W}=1$ and
\begin{eqnarray}\label{Initial_Eigen_Problem_Adjoint}
a(\mathbf\Psi_{h_1}, \mathbf U^{h_1,*})&=&\lambda^{h_1}b(\mathbf\Psi_{h_1}, \mathbf U^{h_1,*}),
\ \ \ \ \forall \mathbf\Psi_{h_1}\in \mathbb V_{h_1}.
\end{eqnarray}

Choose $2q$ eigenpairs $\{\lambda_{j}^{h_1},\mathbf U_{j}^{h_1}\}_{j=i}^{i+q-1}$ and
$\{\lambda_{j}^{h_1},\mathbf U_{j}^{h_1,*}\}_{j=i}^{i+q-1}$
which approximate the desired
eigenvalue $\lambda_i$ and its geometric eigenspaces of the eigenvalue problem
 (\ref{Initial_Eigen_Problem})
and its adjoint one (\ref{Initial_Eigen_Problem_Adjoint}).
Based on these two geometric eigenspace, we compute the
corresponding algebraic eigenspaces $M_{h_1}(\lambda_i):={\rm space}
\big\{\mathbf U_{i}^{h_1},\cdots,\mathbf U_{i+m-1}^{h_1}\big\}$
and $M^*_{h_1}(\lambda_i):={\rm space}\big\{\mathbf U_{i}^{h_1,*},\cdots,
 \mathbf U_{i+m-1}^{h_1,*}\big\}$.
Then do the following correction steps.

\item Construct a series of finer finite element
spaces $\mathbb V_{h_2},\cdots,\mathbb V_{h_n}$ on the sequence of
 nested meshes $\mathcal{T}_{h_2},\cdots,\mathcal{T}_{h_n}$
(c.f. \cite{BrennerScott,Ciarlet}).

\item Do $\ell=1,\cdots,n-1$\\
Obtain new eigenpair approximations
$(\{\lambda_{j}^{h_{\ell+1}}\}_{j=i}^{i+m-1},M_{h_{\ell+1}}(\lambda_i),M^*_{h_{\ell+1}}(\lambda_i))$
by Algorithm \ref{Correction_Step}
\begin{eqnarray*}
&&\big(\{\lambda_{j}^{h_{\ell+1}}\}_{j=i}^{i+m-1},M_{h_{\ell+1}}(\lambda_i),M^*_{h_{\ell+1}}(\lambda_i)\big)=\nonumber\\
&&\ \ \ \ \quad\quad {\it
Correction}\big(\mathbb V_H,\{\lambda_{j}^{h_{\ell}}\}_{j=i}^{i+m-1},M_{h_{\ell}}(\lambda_i),M^*_{h_{\ell}}(\lambda_i),\mathbb V_{h_{\ell+1}}\big).
\end{eqnarray*}
End Do
\end{enumerate}
Finally, we obtain eigenpair approximations
$\big(\{\lambda_{j}^{h_n}\}_{j=i}^{i+m-1},M_{h_n}(\lambda_i),M^*_{h_n}(\lambda_i)\big)$.
\end{algorithm}
\begin{theorem}\label{Error_Estimate_Final_Theorem}
After implementing Algorithm \ref{Multi_Correction}, the resultant
eigenpair approximations $(\{\lambda_{j}^{h_n}\}_{j=i}^{i+m-1},M_{h_n}(\lambda_i),M^*_{h_n}(\lambda_i))$
 have the following error estimates
\begin{eqnarray}
\Theta(M(\lambda_i),M_{h_n}(\lambda_i)) &\lesssim& \varepsilon_{h_n}(\lambda_i),\label{Multi_Correction_Err_fun}\\
\Phi(M(\lambda_i),M_{h_n}(\lambda_i)) &\lesssim&\eta_a^*(H) \varepsilon_{h_n}(\lambda_i),\label{Multi_Correction_Err_fun_Weak}\\
\Theta(M^*(\lambda_i),M^*_{h_n}(\lambda_i)) &\lesssim& \varepsilon^*_{h_n}(\lambda_i),\label{Multi_Correction_Err_fun_Adjoint}\\
\Phi(M^*(\lambda_i),M^*_{h_n}(\lambda_i)) &\lesssim&\eta_a(H) \varepsilon^*_{h_n}(\lambda_i),\label{Multi_Correction_Err_fun_Weak_Adjoint}\\
|\widehat{\lambda}_{i}^{h_n}-\lambda_i|&\lesssim&\varepsilon_{h_n}(\lambda_i)\varepsilon^*_{h_n}(\lambda_i),
\label{Multi_Correction_Err_eigen}
\end{eqnarray}
where $\widehat{\lambda}_{i}^{h_n}=\frac{1}{m}\sum_{j=i}^{i+m-1}\lambda_{j}^{h_n}$,
$\varepsilon_{h_n}(\lambda_i)=\sum_{k=1}^{n}\eta_a^*(H)^{n-k}\delta_{h_{\ell}}(\lambda_i)$ and\\
$\varepsilon^*_{h_n}(\lambda_i)=\sum_{k=1}^{n}\eta_a(H)^{n-k}\delta^*_{h_{\ell}}(\lambda_i)$.
\end{theorem}
\begin{proof}
First, we set $\varepsilon_{h_1}(\lambda_i):= \delta_{h_1}(\lambda_i)$ and
 $\varepsilon_{h_1}^*(\lambda_i):=\delta_{h_1}^*(\lambda_i)$.
Then the following estimates hold
\begin{eqnarray}
\Theta(M(\lambda_i),M_{h_1}(\lambda_i)) &\lesssim& \varepsilon_{h_1}(\lambda_i),\label{Multi_Correction_Err_fun_h_1}\\
\Phi(M(\lambda_i),M_{h_1}(\lambda_i)) &\lesssim&\eta_a^*(h_1) \varepsilon_{h_1}(\lambda_i)\leq \eta_a^*(H) \varepsilon_{h_1}(\lambda_i),\label{Multi_Correction_Err_fun_Weak_h_1}\\
\Theta(M^*(\lambda_i),M^*_{h_1}(\lambda_i)) &\lesssim& \varepsilon^*_{h_1}(\lambda_i),\label{Multi_Correction_Err_fun_Adjoint_h_1}\\
\Phi(M^*(\lambda_i),M^*_{h_1}(\lambda_i)) &\lesssim&\eta_a(h_1) \varepsilon^*_{h_1}(\lambda_i)\leq \eta_a(H) \varepsilon^*_{h_1}(\lambda_i).\label{Multi_Correction_Err_fun_Weak_Adjointh_1}
\end{eqnarray}

By recursive relation and Theorem \ref{Error_Estimate_One_Correction_Theorem},
{ we derive}
\begin{eqnarray}\label{epsilon_n_1}
\Theta(M(\lambda_i),M_{h_n}(\lambda_i))&\lesssim&\varepsilon_{h_n}(\lambda_i)
= \eta_a^*(H)\varepsilon_{h_{n-1}}(\lambda_i)
+\delta_{h_n}(\lambda_i)\nonumber\\
&\lesssim&\eta_a^*(H)^2\varepsilon_{h_{n-2}}(\lambda_i)+
\eta_a^*(H)\delta_{h_{n-1}}(\lambda_i)+\delta_{h_n}(\lambda_i)\nonumber\\
&\lesssim&\sum\limits_{k=1}^n\eta_a^*(H)^{n-k}\delta_{h_{\ell}}(\lambda_i)
\end{eqnarray}
and
\begin{eqnarray}\label{Error_n-1_Negative_norm}
\Phi(M(\lambda_i),M_{h_n}(\lambda_i))&\lesssim& \eta_a^*(H)\sum\limits_{k=1}^n\eta_a^*(H)^{n-k}\delta_{h_{\ell}}(\lambda_i).
\end{eqnarray}
These are the estimates (\ref{Multi_Correction_Err_fun}) and (\ref{Multi_Correction_Err_fun_Weak}) and the
estimates (\ref{Multi_Correction_Err_fun_Adjoint}) and (\ref{Multi_Correction_Err_fun_Weak_Adjoint}) can be proved similarly.
From Theorem \ref{Error_Estimate_Theorem}, (\ref{Multi_Correction_Err_fun}) and (\ref{Multi_Correction_Err_fun_Adjoint}),
we can obtain the estimate (\ref{Multi_Correction_Err_eigen}).
\end{proof}
In order to give the final error estimate results for the eigenpair approximations by the multilevel correction method,
we assume the following properties for the error estimates hold \cite{BrennerScott,Ciarlet,Shaidurov}
\begin{eqnarray}\label{Relations}
\delta_{h_{\ell}+1}(\lambda_i)\approx \frac{1}{\beta}\delta_{h_{\ell}}(\lambda_i),
\ \ \  \delta_{h_{\ell}+1}^*(\lambda_i)\approx \frac{1}{\beta}\delta_{h_{\ell}}^*(\lambda_i)
\end{eqnarray}
when the mesh sizes $h_{\ell}$, $h_{\ell+1}$ satisfy the relation (\ref{Mesh_Size_Relation}) and
the eigenfunctions have the corresponding regularities.
\begin{corollary}
After implementing Algorithm \ref{Multi_Correction}, the resultant
eigenpair approximations $(\{\lambda_{j}^{h_n}\}_{j=i}^{i+m-1},M_{h_n}(\lambda_i),M^*_{h_n}(\lambda_i))$
 have the following error estimates
\begin{eqnarray}
\Theta(M(\lambda_i),M_{h_n}(\lambda_i)) &\lesssim& \delta_{h_n}(\lambda_i),\label{Multi_Correction_Err_fun_2}\\
\Phi(M(\lambda_i),M_{h_n}(\lambda_i)) &\lesssim&\eta_a^*(H) \delta_{h_n}(\lambda_i),\label{Multi_Correction_Err_fun_Weak_2}\\
\Theta(M^*(\lambda_i),M^*_{h_n}(\lambda_i)) &\lesssim& \delta^*_{h_n}(\lambda_i),\label{Multi_Correction_Err_fun_Adjoint_2}\\
\Phi(M^*(\lambda_i),M^*_{h_n}(\lambda_i)) &\lesssim&\eta_a(H) \delta^*_{h_n}(\lambda_i),\label{Multi_Correction_Err_fun_Weak_Adjoint_2}\\
|\widehat{\lambda}_{i}^{h_n}-\lambda_i|&\lesssim&\delta_{h_n}(\lambda_i)\delta^*_{h_n}(\lambda_i),
\label{Multi_Correction_Err_eigen_2}
\end{eqnarray}
when the mesh size $H$ is small enough, the conditions (\ref{Relations}), $C\beta\eta_a^*(H) <1$ and $C\beta\eta_a(H)<1$ hold for the hidden constant $C$.
\end{corollary}
\begin{proof}
When the mesh size $H$ is small enough, the conditions (\ref{Relations}) and $C\beta\eta_a^*(H) <1$ hold,
we have the following inequalities
\begin{eqnarray*}
\varepsilon_{h_n}(\lambda_i)&=&\sum_{k=1}^{n}\eta_a^*(H)^{n-k}\delta_{h_{\ell}}(\lambda_i)
\approx\left(\sum_{k=1}^{n}\Big(\eta_a^*(H)\beta\Big)^{n-k}\right)\delta_{h_n}(\lambda_i)\nonumber\\
&\lesssim& \frac{1}{1-C\beta\eta_a^*(H)}\delta_{h_n}(\lambda_i)\lesssim \delta_{h_n}(\lambda_i).
\end{eqnarray*}
Combining the above estimate and Theorem \ref{Error_Estimate_Final_Theorem}, we can obtain the desired results
(\ref{Multi_Correction_Err_fun_2}) and (\ref{Multi_Correction_Err_fun_Weak_2}). Similar argument can lead
to the results (\ref{Multi_Correction_Err_fun_Adjoint_2}) and (\ref{Multi_Correction_Err_fun_Weak_Adjoint_2}).
Then the result (\ref{Multi_Correction_Err_eigen_2}) can be derived from (\ref{Multi_Correction_Err_fun_2})-(\ref{Multi_Correction_Err_fun_Weak_Adjoint_2}) and the proof is complete.
\end{proof}

\section{Numerical results}\label{Numerical_Results}

In this section, we present four examples to validate the efficiency of the proposed multilevel
correction scheme defined by Algorithm \ref{Multi_Correction}. The conforming linear finite element
 is used in the discretization for all the examples.

%
%
%
%

\subsection{Transmission eigenvalue problem on the unit disk}
Let $\Omega$ be a unit disk, $A=aI$ for some constant $a> 1$ and $n> 1$ is also a constant.
The solutions of \eqref{Problem} can be written as
\begin{equation}\label{vdisk}
 w = J_m(kr\sqrt{n/a})\cos(m\theta) ,\quad
 v = J_m(kr)\cos(m\theta) ,\quad m=0,1,2,\cdots,
\end{equation}
where $J_m(z)$ is the first kind Bessel function of order $m$.
In order to satisfy the boundary condition $w=v$ on $\partial\Omega$, one can choose
\begin{equation}\label{wdisk}
w = \frac{J_m(k)}{J_m(k\sqrt{n/a})} J_m\big(kr\sqrt{n/a}\big)\cos(m\theta),\quad m=0,1,2,\cdots.
\end{equation}
Of course, the trigonometric functions in \eqref{vdisk}-\eqref{wdisk} can also be chosen as
 $\sin(m\theta)$. Ignoring the trigonometric functions and using the recursive identity
\begin{equation*}
\frac{d J_m(z)}{d z} =  \frac{m}{z}J_m(z) - J_{m+1}(z) ,
\end{equation*}
we obtain that
\begin{eqnarray*}
&& \frac{\partial w}{\partial\nu_A} 
= a\frac{J_m(k)}{J_m(k\sqrt{n/a})} k \sqrt{\frac{n}{a}} \left( \frac{m}{kr\sqrt{n/a}}J_m\big(kr\sqrt{n/a}\big)
- J_{m+1}\big(kr\sqrt{n/a}\big) \right) ,\\
&& \frac{\partial v}{\partial\nu} 
= k \left( \frac{m}{kr}J_m(kr) - J_{m+1}(kr) \right).
\end{eqnarray*}
Thus the boundary condition for the normal derivative implies that
\begin{eqnarray}
\nonumber
&&J_{m+1}(k)J_m\big(k\sqrt{n/a}\big) - \sqrt{na}J_m(k)J_{m+1}\big(k\sqrt{n/a}\big)\\
\label{eqn-k}
&&\qquad\qquad\qquad\qquad\quad +\ (a-1)\frac{m}{k}J_m(k)J_m\big(k\sqrt{n/a}\big) = 0.
\end{eqnarray}
From \eqref{eqn-k}, we can get the exact transmission eigenvalues and the
corresponding eigenfunctions by using \eqref{vdisk}-\eqref{wdisk}.  In this example, we take $a=2$ and $n=8$.
The eigenvalues have multiplicity 2 for $m > 0$ due to the trigonometric terms $\cos(m\theta)$ and $\sin(m\theta)$.



We use the quasi-uniform meshes $\{\mathcal{T}_{h_\ell}\}_{\ell}$ and solve the transmission eigenvalue problem by using Algorithm \ref{Multi_Correction}. Let $k_{j,\ell}$ be the $j$-th transmission eigenvalue computed
on the mesh $\mathcal{T}_{h_{\ell}}$, while $k_j$ be the $j$-th exact eigenvalue obtained from \eqref{eqn-k}.
The left part of Figure \ref{fig-circle} shows the $\log N_{\ell}$-$\log \sum_{j=1}^6 |k_{j,\ell}-k_j|$ curve,
where $N_{\ell}$ is the number of degrees of freedom(DOFs) with the mesh $\mathcal{T}_{h_{\ell}}$ which is almost double of the number of nodes.
It indicates that the sum of the errors for the first six transmission eigenvalues decreases
as $O(N_{\ell}^{-1})$ or has the $O(h^2)$ convergence order as implied in Theorem \ref{Error_Estimate_Final_Theorem}.
Table \ref{tab-circle} depicts the first six transmission eigenvalues
computed by Algorithm \ref{Multi_Correction} on the finest mesh.

\begin{figure}[htp]
\caption{Eigenvalue and eigenfunction errors on the unit disk}\label{fig-circle}
\begin{center}
\includegraphics[width=6cm, height=4.5cm]{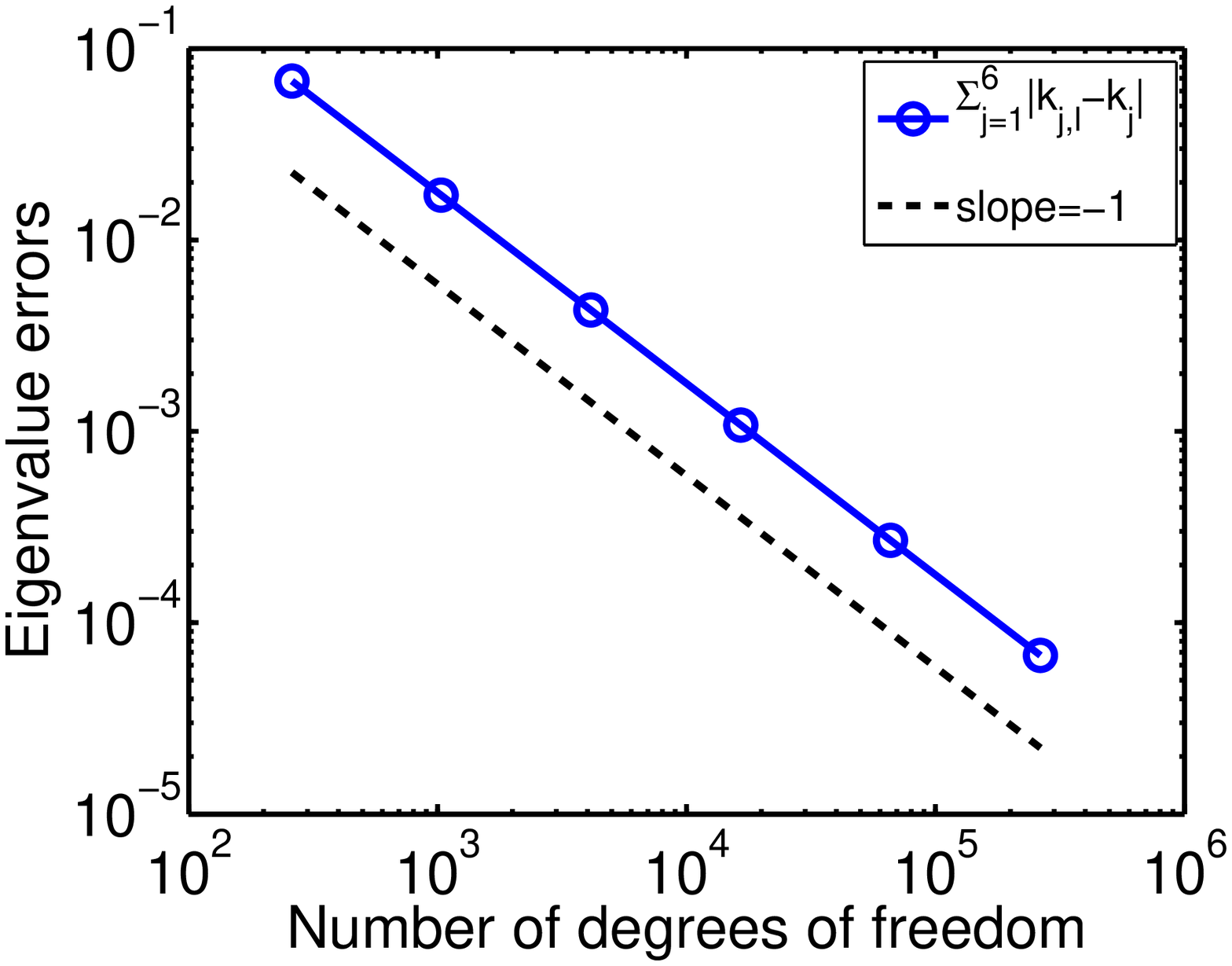}
\includegraphics[width=6cm, height=4.5cm]{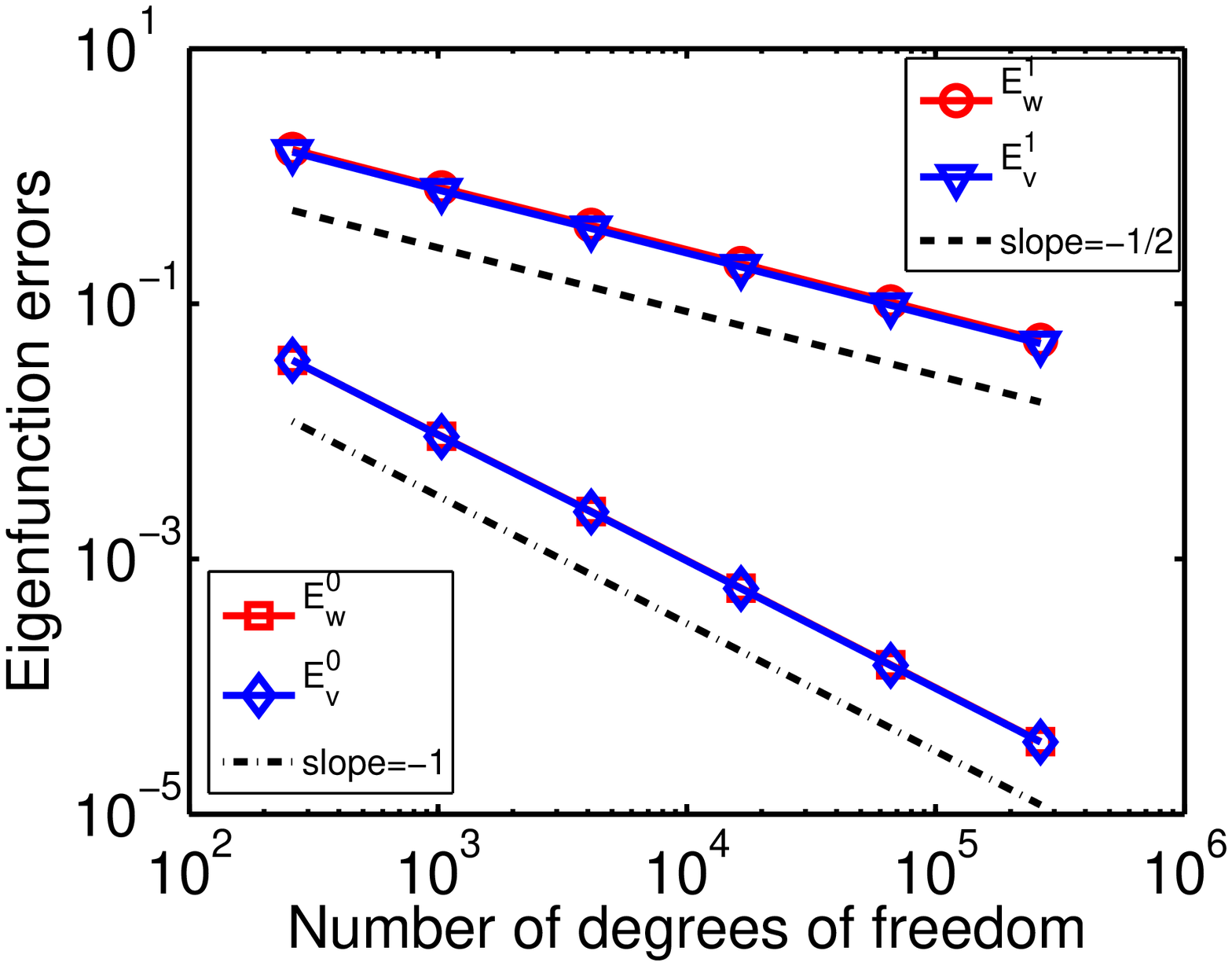}
\end{center}
\end{figure}

\begin{table}[htp]
\caption{The first six transmission eigenvalues computed on the unit disk.}
\label{tab-circle}
\setlength{\tabcolsep}{10pt}
\begin{center}
\begin{tabular}{cccccccc}\hline
$N_\ell$  & $k_{1,\ell}$ & $k_{2,\ell}$  &  $k_{3,\ell}$  &  $k_{4,\ell}$  &  $k_{5,\ell}$  &  $k_{6,\ell}$\\ \hline
264194 & 0.7176 & 0.7176 & 1.2106  & 1.2106  &  1.6841  &  1.6841  \\ \hline
\end{tabular}
\end{center}
\end{table}

Since the eigenvalues have multiplicity $2$ for $m>0$, we compute the distance to the eigenspaces, i.e.
\begin{equation}
E_u^s = \sum_{j=1}^6\min_{\alpha\in\mathcal{C}^2}\left\| u_{j,\ell} - \sum_{i=1}^2\alpha_i u_{j_i}\right\|_{s},
 \ \ u=w,v,\ s=0,1,
\end{equation}
where $u_{j_1}, u_{j_2}$ are the corresponding eigenfunctions to the eigenvalue $k_j$.
The right part of Figure \ref{fig-circle} shows the eigenfunction errors versus the number
 of elements $N_{\ell}$. It is observed that the $H^1$-error decreases as $O(N_{\ell}^{-1/2})$
 or has the $O(h)$ convergence order, and the $L^2$-error decreases as $O(N_{\ell}^{-1})$ or has
  the $O(h^2)$ convergence order as implied in Theorem \ref{Error_Estimate_Final_Theorem}.


%
%

\subsection{Transmission eigenvalue problem on the unit square}

In this subsection, let $\Omega=(0,1)^2$ be the unit square and
\begin{equation*} 
A(x)=\begin{pmatrix}
2+x_1^2 & x_1x_2\\
x_1x_2 & 2+x_2^2
\end{pmatrix},\quad n(x)= 4+2(x_1+x_2).
\end{equation*}
It is easy to verify that the symmetric matrix $A(x)$ and the index of refraction
$n(x)$ satisfy the condition \eqref{cond-A-n}.

\begin{figure}[htp]
\caption{Eigenvalue and eigenfunction errors on the unit square}
\label{fig-square}
\begin{center}
\includegraphics[width=6cm, height=4.5cm]{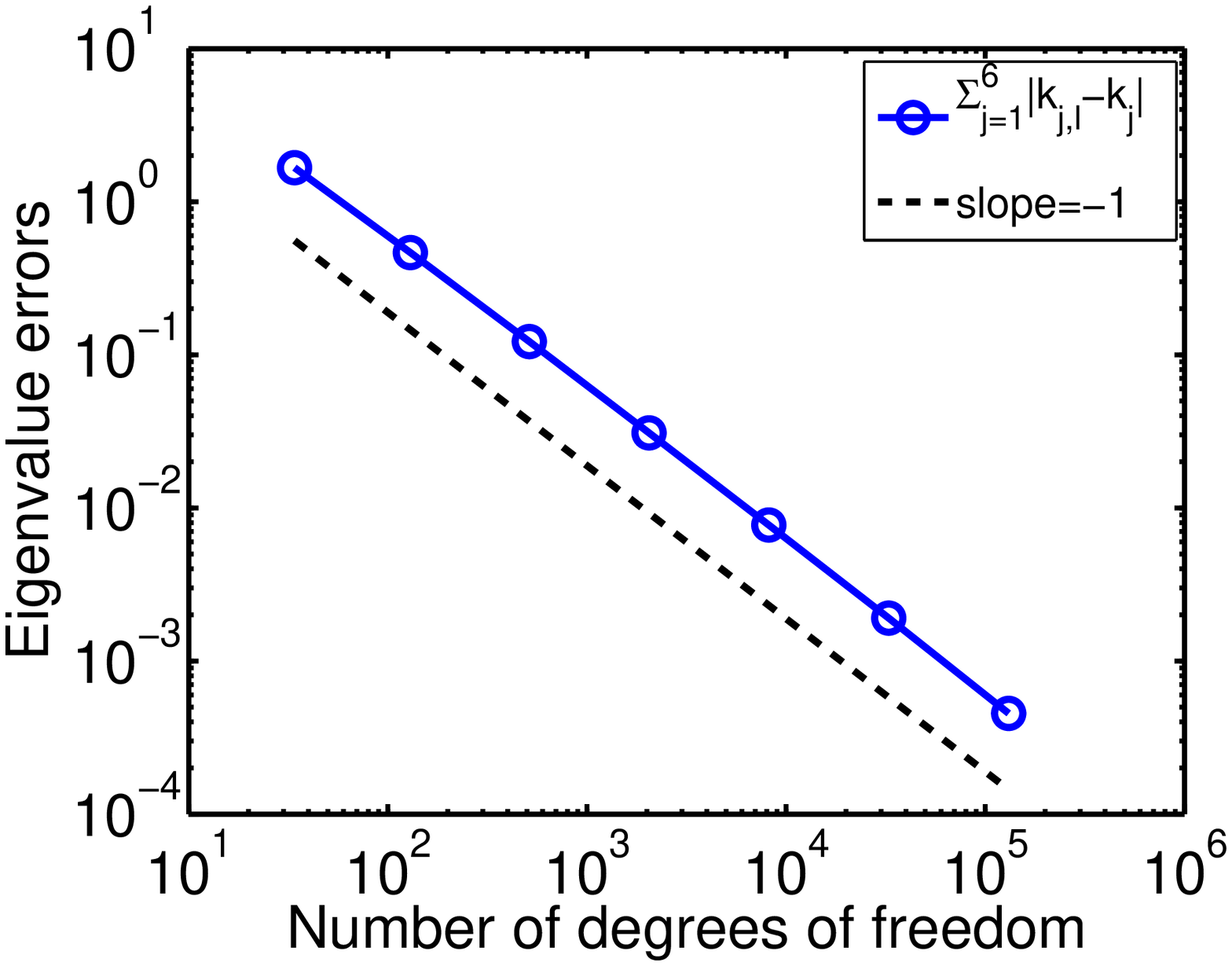}
\includegraphics[width=6cm, height=4.5cm]{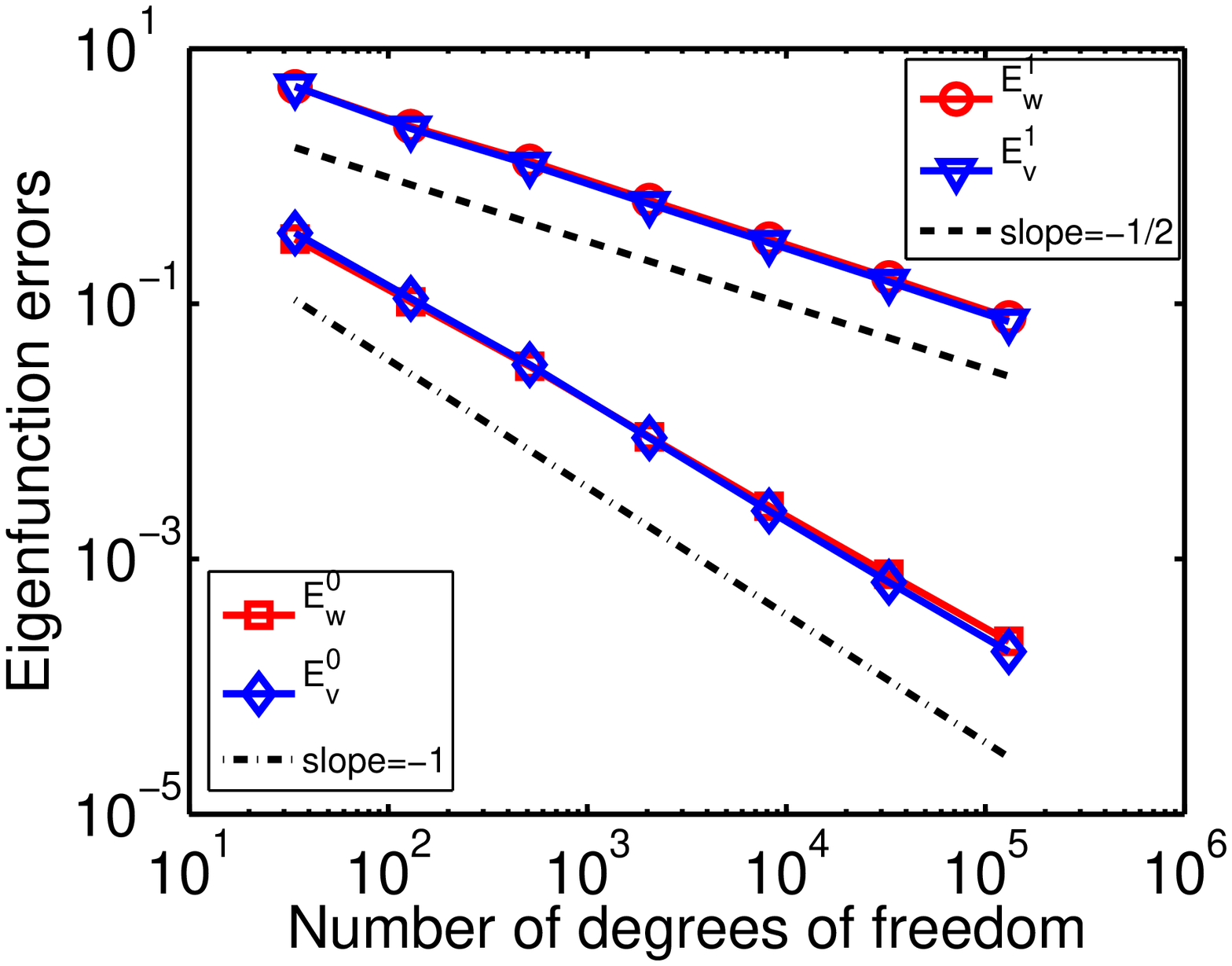}
\end{center}
\end{figure}

The quasi-uniform meshes $\{\mathcal{T}_{h_{\ell}}\}_{\ell}$ are used in Algorithm \ref{Multi_Correction}
to solve the transmission eigenvalue problem. Let $k_{j,l}$ be the $j$-th transmission eigenvalue computed on the mesh $\mathcal{T}_{h_{\ell}}$. In addition, let $k_j$ be $j$-th `exact' eigenvalue obtained numerically
 on a very fine mesh with the number of DOFs $N_{\ell}>10^6$.
The left part of Figure \ref{fig-square} shows the $\log N_{\ell}$-$\log\sum_{i=1]}^6 |k_{j,\ell}-k_j|$ curve,
 which again indicates that the sum of the errors for the first six transmission eigenvalues
 decreases as $O(N_{\ell}^{-1})$ or has the $O(h^2)$ convergence order.
Table \ref{tab-square} depicts the first six transmission eigenvalues computed by
Algorithm \ref{Multi_Correction} on the mesh with $N_{\ell}=131074$.

\begin{table}[htp]
\caption{The first six transmission eigenvalues computed on the unit square.}
\label{tab-square}
\setlength{\tabcolsep}{10pt}
\begin{center}
\begin{tabular}{cccccccc}\hline
$N_\ell$  & $k_{1,\ell}$ & $k_{2,\ell}$  &  $k_{3,\ell}$  &  $k_{4,\ell}$  &  $k_{5,\ell}$  &  $k_{6,\ell}$\\ \hline
131074 & 1.4808 & 1.7425 & 2.3340  & 3.1636  &  3.6559  &  3.7665  \\ \hline
\end{tabular}
\end{center}
\end{table}

The right part of Figure \ref{fig-square} shows the eigenfunction errors
\begin{equation}
E_u^s = \sum_{j=1}^6\|u_{j,\ell} - u_j\|_{s},\ \ u=w,v,\ s=0,1.
\end{equation}
It is also observed that the $H^1$-error decreases as $O(N_{\ell}^{-1/2})$ or has the $O(h)$ convergence order,
and the $L^2$-error decreases as $O(N_{\ell}^{-1})$ or has the $O(h^2)$ convergence order as implied in Theorem \ref{Error_Estimate_Final_Theorem}.


%
%
%

\subsection{Transmission eigenvalue problem on the unit square with other conditions on $A(x)$ and $n(x)$}

In this subsection, let $\Omega=(0,1)^2$ be the unit square.
It is stated in Remark \ref{Remark_Condition}
that our algorithm behaves well if the condition \eqref{cond-A-n} is replaced by the condition \eqref{cond-A-n-2}.
Actually, by using similar arguments, we can also get the convergence results under
the condition \eqref{cond-A-n-2}.
Let
\begin{equation*} 
A(x)=\begin{pmatrix}
\frac12+\frac18 x_1^2 & \frac18 x_1x_2\\
\frac18 x_1x_2 & \frac12+\frac18 x_2^2
\end{pmatrix},\quad n(x)= \frac14+\frac18(x_1+x_2).
\end{equation*}
It is easy to verify that the symmetric matrix $A(x)$ and the index of refraction $n(x)$ satisfy the condition \eqref{cond-A-n-2}.

\begin{figure}[htp]
\caption{Eigenvalue and eigenfunction errors on the unit square with other conditions on $A(x)$ and $n(x)$}
\label{fig-squareB}
\begin{center}
\includegraphics[width=6cm, height=4.5cm]{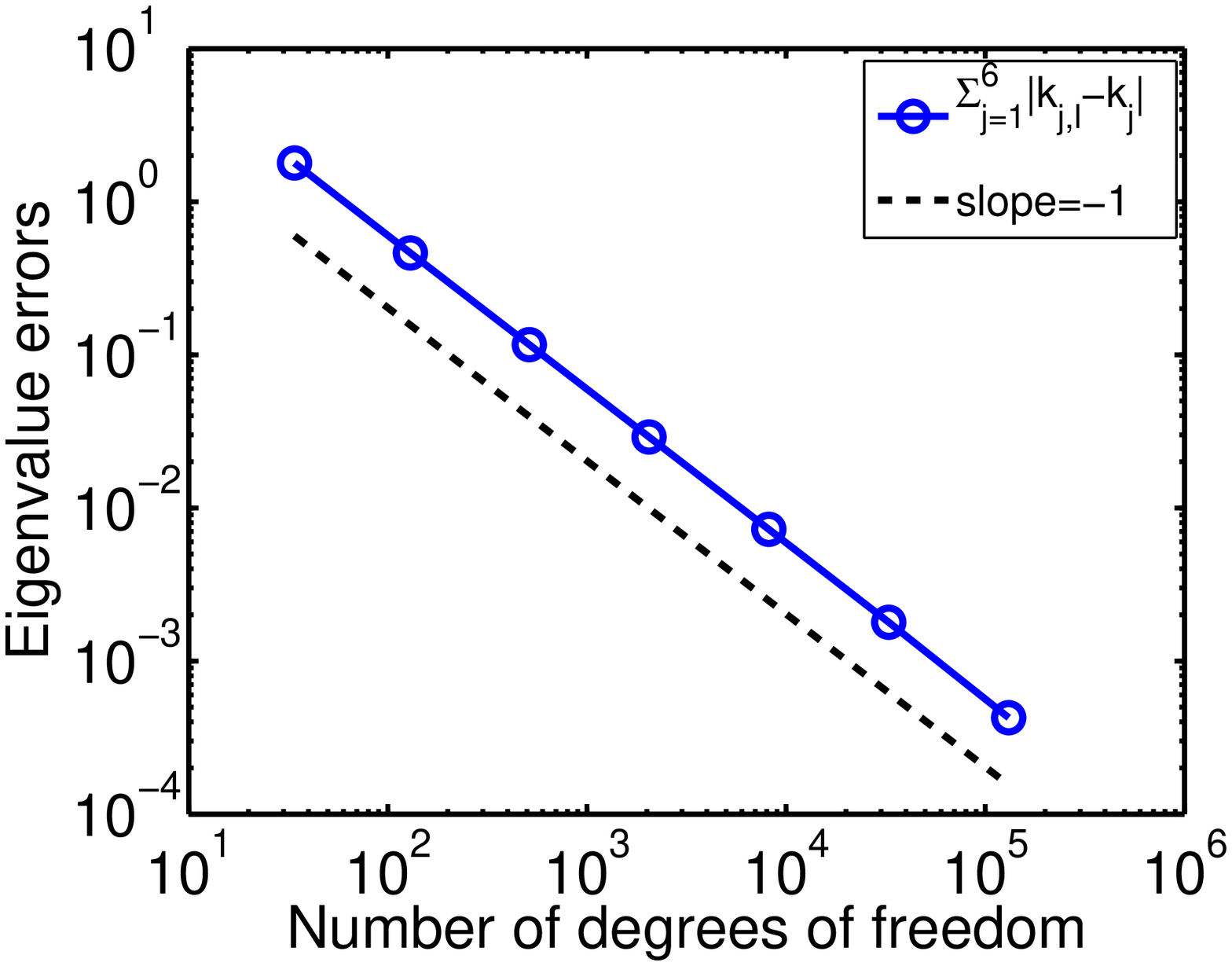}
\includegraphics[width=6cm, height=4.5cm]{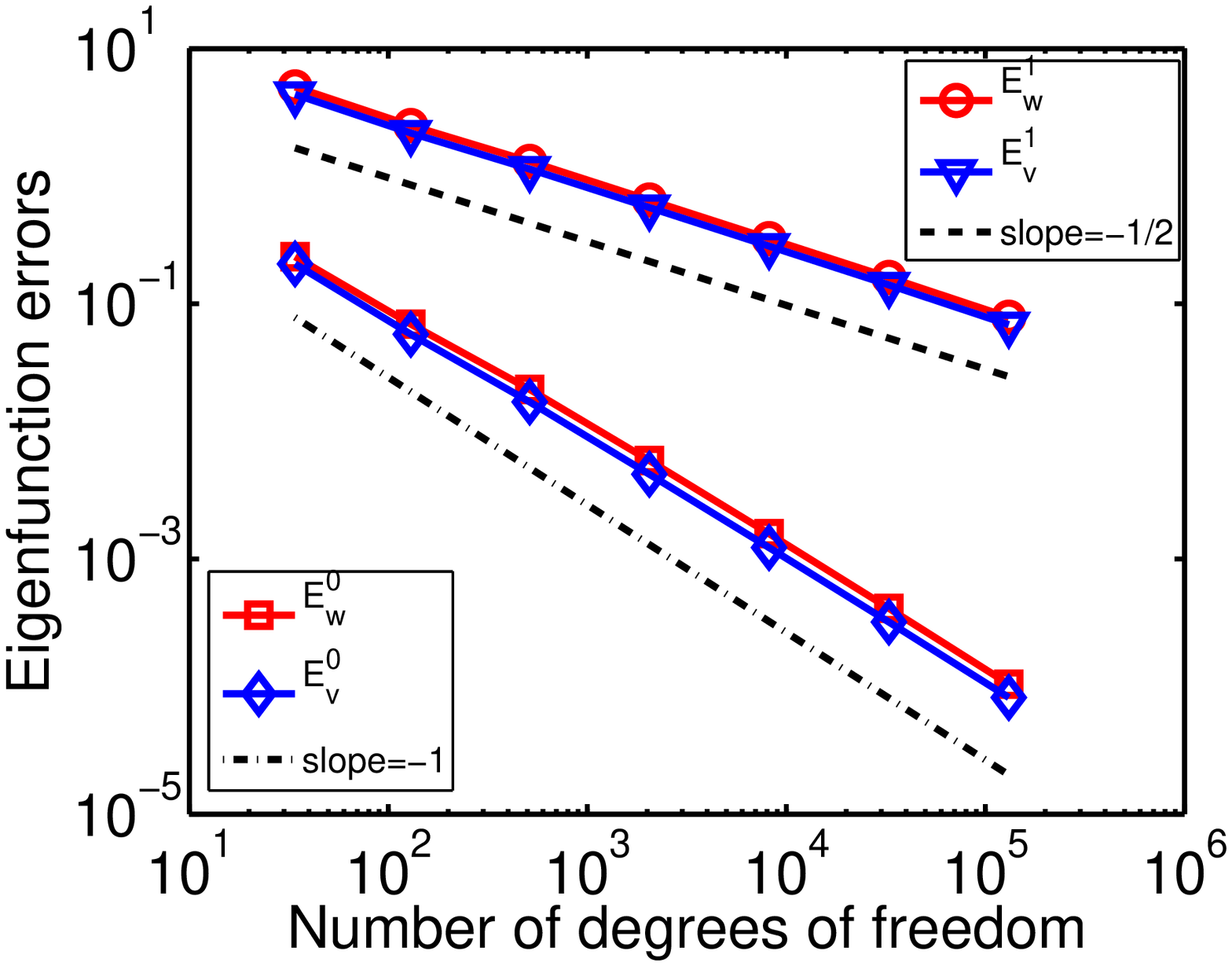}
\end{center}
\end{figure}

By using the same setting as in the previous subsection, we show the eigenvalue errors and eigenfunction errors in the left part and the right part of Figure \ref{fig-squareB}, respectively.
It again demonstrates that the convergence results in Theorem \ref{Error_Estimate_Final_Theorem} is also valid for this kind of coefficient matrix $A(x)$ and the index of refraction $n(x)$. In addition, the first six transmission eigenvalues computed by Algorithm \ref{Multi_Correction} are depicted in Table \ref{tab-squareB}.

\begin{table}[htp]
\caption{The first six transmission eigenvalues computed on the unit square.}
\label{tab-squareB}
\setlength{\tabcolsep}{10pt}
\begin{center}
\begin{tabular}{cccccc}\hline
$N_\ell$ & $k_{1,\ell}$ & $k_{2,\ell}$  &  $k_{3,\ell}$  &  $k_{4,\ell}$  &  $k_{5,\ell}$,  $k_{6,\ell}$\\ \hline
131074 & 2.6786 & 2.7995 & 3.8921  & 5.5341  &  5.8252 $\pm$ 0.8502i  \\ \hline
\end{tabular}
\end{center}
\end{table}


%
%

\subsection{Transmission eigenvalue problem on the L-shape domain}

In this subsection, let $\Omega=(-1,1)^2\backslash[0,1)\times(-1,0]$ be the L-shape domain and
\begin{equation*} 
A(x)=\begin{pmatrix}
2+x_1^2 & x_1x_2\\
x_1x_2 & 2+x_2^2
\end{pmatrix},\quad n(x)= 2+|x_1+x_2|.
\end{equation*}
It is easy to verify that the symmetric matrix $A(x)$ and the index of refraction $n(x)$ satisfy the condition \eqref{cond-A-n}.

\begin{figure}[htp]
\caption{Eigenvalue and eigenfunction errors on the L-shape domain}
\label{fig-lshape}
\begin{center}
\includegraphics[width=6cm, height=4.5cm]{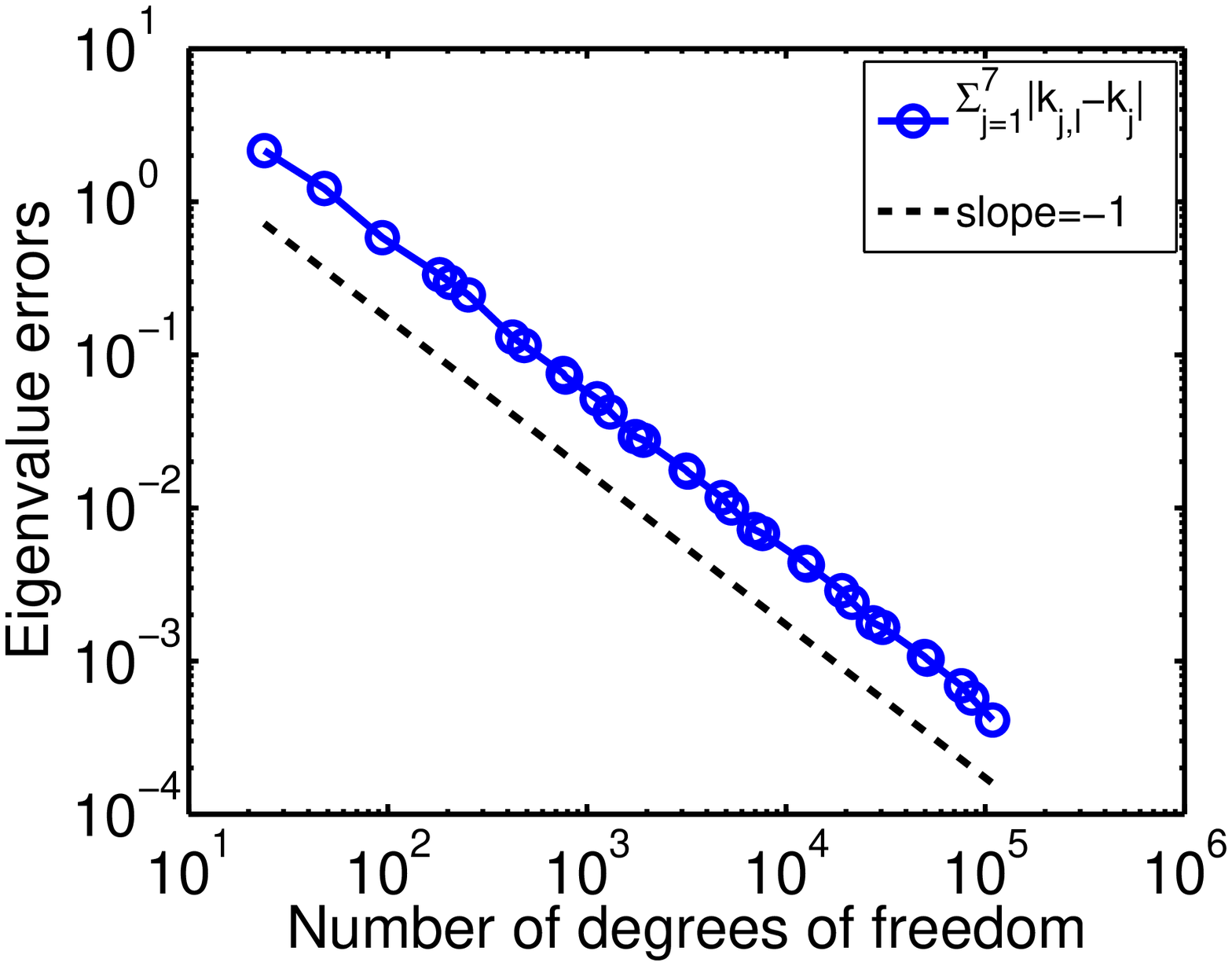}
\includegraphics[width=6cm, height=4.5cm]{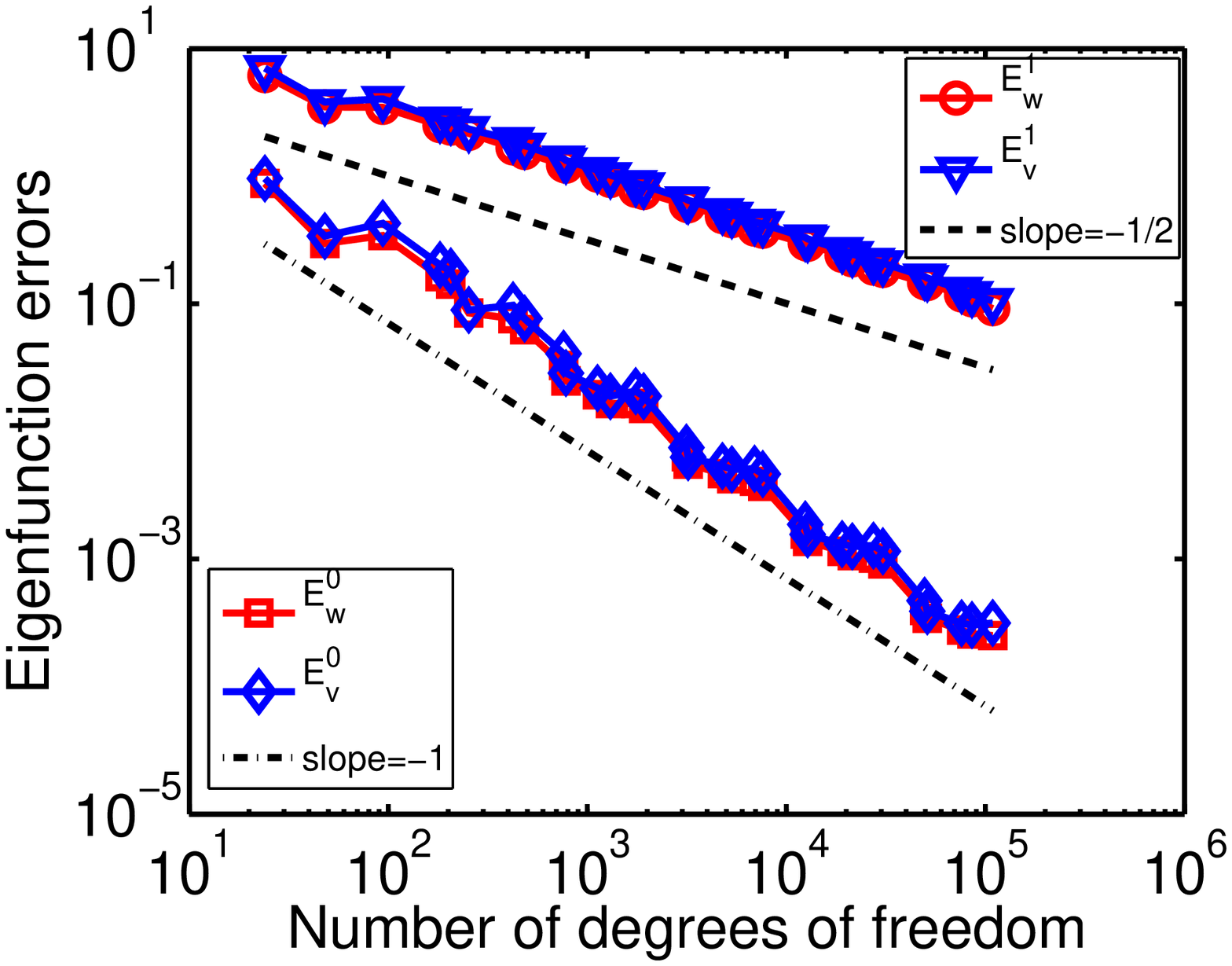}
\end{center}
\end{figure}

Since $\Omega$ has reentrant corner, the singularity of the eigenfunctions is expected.
We turn to apply the adaptive finite element method for the mesh refinement.
Here the ZZ recovery method (c.f. \cite{ZZ}) is adopted as the a posteriori
error estimator for eigenfunction and adjoint eigenfunction approximations.
Similarly, the `exact' eigenvalues are obtained numerically on a very fine
mesh with the number of DOFs $N_{\ell}>10^6$.

Figure \ref{fig-lshape} shows the numerical results for this example.
 Figure \ref{fig-lshape-mesh} shows the initial mesh and the mesh after $17$ adaptive iterations.
It is observed that  the multilevel correction method works well on the adaptive meshes and the computational complexity are also quasi-optimal. In addition, Table \ref{tab-lshape} depicts the first seven transmission eigenvalues computed by Algorithm \ref{Multi_Correction} on the mesh with $N_{\ell}=108914$.

\begin{figure}[htp]
\caption{Initial mesh and the mesh after 17 adaptive iterations}
\label{fig-lshape-mesh}
\begin{center}
\includegraphics[width=6cm, height=4.5cm]{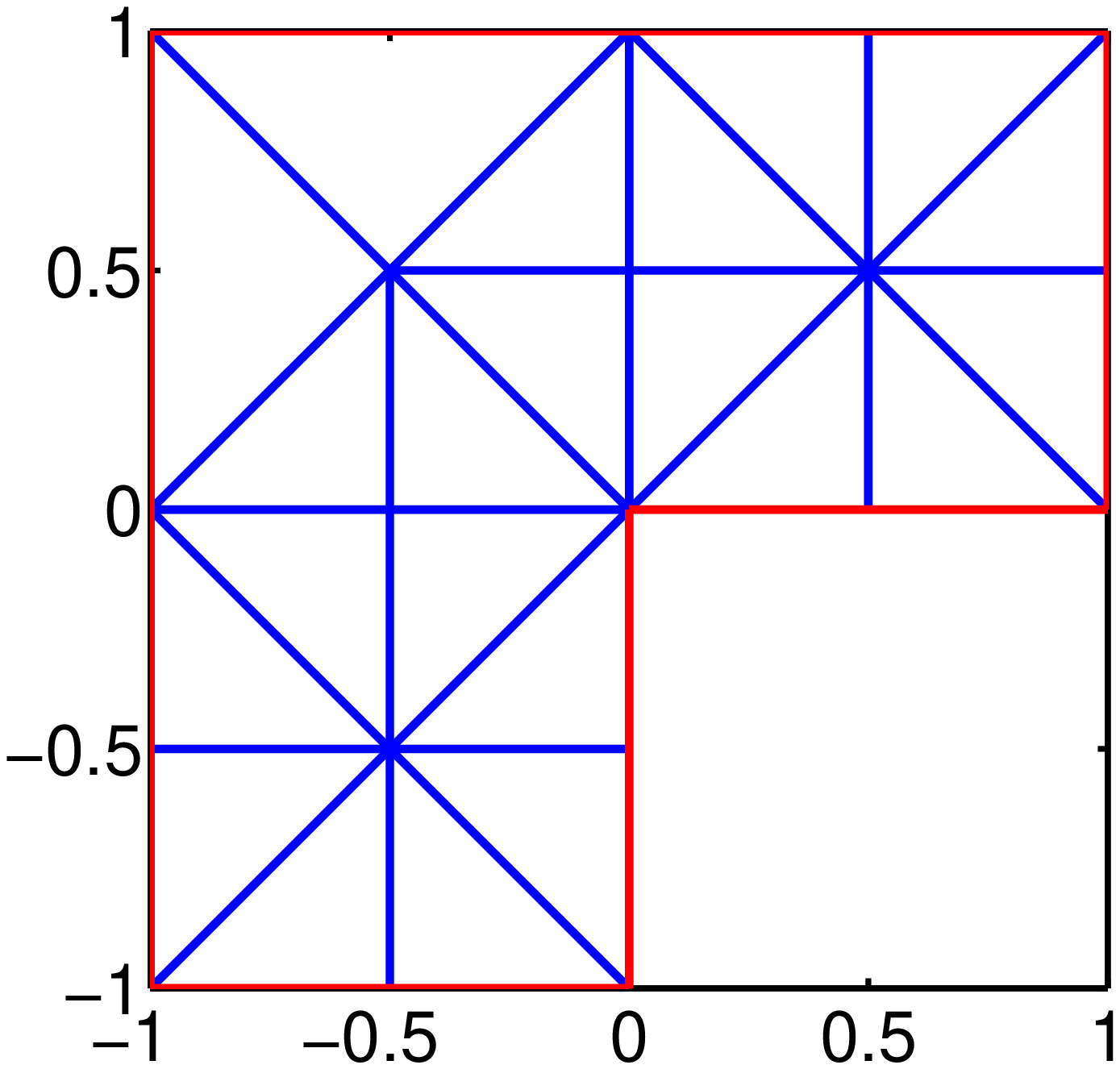}
\includegraphics[width=6cm, height=4.5cm]{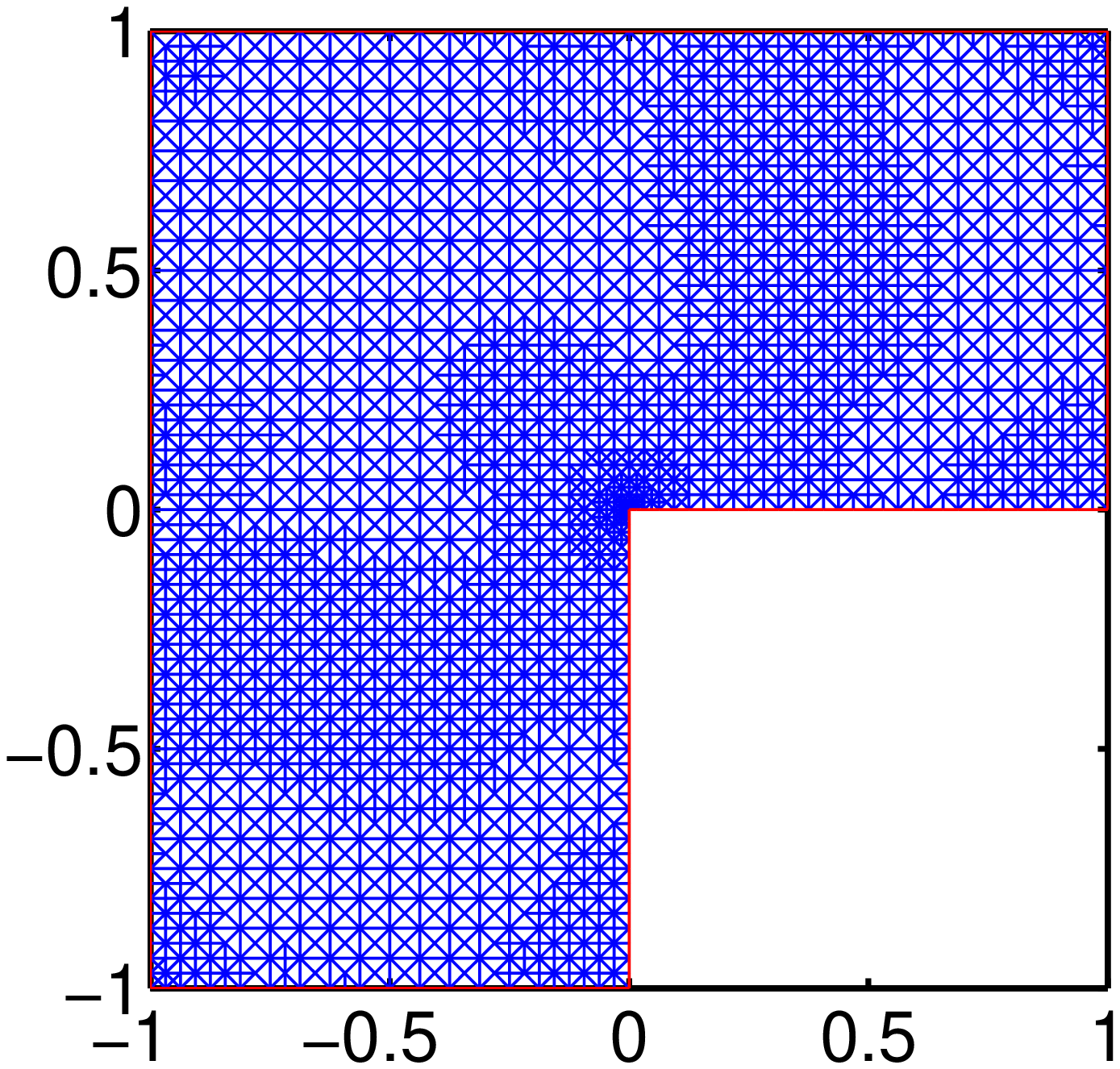}
\end{center}
\end{figure}

\begin{table}[htp]
\caption{The first seven transmission eigenvalues computed on the L-shape domain.}
\label{tab-lshape}
\setlength{\tabcolsep}{6pt}
\begin{center}
\begin{tabular}{ccccccc}\hline
$N_{\ell}$  & $k_{1,\ell}$ & $k_{2,\ell}$  &  $k_{3,\ell}$  &  $k_{4,\ell}$  &  $k_{5,\ell}$  &  $k_{6,\ell}$, $k_{7,\ell}$\\ \hline
108914 & 0.8740 & 1.5895 & 2.4038  & 2.6197  &  2.8764  &  3.0449 $\pm$ 0.0824i  \\ \hline
\end{tabular}
\end{center}
\end{table}


\section{Concluding remarks}
In this paper, we give the error estimates for the transmission eigenvalue problem by the
finite element method. Furthermore, based on the obtained error estimates in Theorem \ref{Error_Estimate_Theorem},
a type of multilevel correction method is proposed to
solve the transmission eigenvalue problem. In the multilevel correction method, we
transform the transmission eigenvalue solving in the finest finite element space
to a sequence of linear problems and some transmission eigenvalue solving
in a very low dimensional space. Since the main computational work is to solve the sequence of
linear problems, the multilevel correction method can improve the overfull efficiency
of the transmission eigenvalue solving. The numerical results also show the
 efficiency of the proposed numerical scheme.

%
\section*{Acknowledgments}
The work of Hehu Xie is supported in part
by the National Natural Science Foundations of China (NSFC 91330202, 11371026,
11001259, 11031006, 2011CB309703),  the National
Center for Mathematics and Interdisciplinary Science
the national Center for Mathematics and Interdisciplinary Science, CAS.
The work of Xinming Wu is supported in part by the National Natural Science Foundations of China (NSFC 91330202, 11301089).



\section*{References}
\begin{thebibliography}{99}

\bibitem{AnShen}
J. An  and J. Shen, A spectral-element method for transmission eigenvalue problems,
J. Sci. Comput., 57(3) (2013), 670-688.


\bibitem{BabuskaOsborn}
I. Babu\v{s}ka and J. E. Osborn, Eigenvalue problems. In: P.G. Ciarlet, J.L. Lions (eds.)
Handbook of Numerical Analysis, Vol. II, Finite Element Methods (Part 1), pp. 641-787. North-Holland, Amsterdam, 1991.


\bibitem{BonnetChesnelHaddar}
A. S. Bonnet-Ben Dhia, L. Chesnel and H. Haddar, On the use of
$\mathbb T$-coercivity to study the interior transmission eigenvlaue
problem, C. R. Acad. Sci. Paris, Ser. I, 349 (2011), 647-651.

\bibitem{BonnetChesnelCiarlet}
A. S. Bonnet-Ben Dhia, L. Chesnel and P. J. Ciarlet, $\mathbb T$-coercivity for scalar interface
 problems between dielectrics and metamaterials, ESAIM: M2AN, 46 (2012), 1363-1387.

\bibitem{ChesnelCiarlet}
L. Chesnel and P. J. Ciarlet, $\mathbb T$-coercivity and continuous Galerkin methods:
application to transmission problems with sign changing coefficients, Numer. Math., 124 (2013), 1-29.

\bibitem{BonnetCiarletZwolf}
A. S. Bonnet-Ben Dhia, P. J. Ciarlet and C. M. Zw\"{o}lf, Time harmonic wave diffraction problems in materials
with sign-shifting coefficients, J. Comput. Appl. Math., 234 (2010), 1912-1919.


\bibitem{BrennerScott}
S. Brenner and L. Scott, The Mathematical Theory of Finite Element
Methods, Springer-Verlag, New York, 1994.

\bibitem{CakoniCayorenColton}
F. Cakoni, M. \c{C}Ay\"{o}ren and D. Colton, Transmission eigenvalues and the nondestructive
testing of dielectrics, Inverse Probl., 24 (2008), 065016.

\bibitem{CakoniColtonMonkSun}
F. Cakoni, D. Colton, P. Monk and J. Sun, The inverse electromagnetic scattering problem
for anisotropic media, Inverse Probl., 26 (2010), 074004.

\bibitem{CakoniGintidesHaddar}
F. Cakoni, D. Gintides and H. Haddar, The existence of an infinite discrete set of transmission
eigenvalues, SIAM J. Math. Anal., 42 (2010), 237-255.

\bibitem{CakoniHaddar}
F. Cakoni and H. Haddar, Transmission eigenvalues in inverse scattering theory,
Inside Out II, G. Uhlmann editor, MSRI Publications, 60 (2012), 526-578.


\bibitem{Ciarlet}
P. G. Ciarlet, The Finite Element Method for Elliptic Problems, Classics Appl. Math.,
SIAM Philadelphia, 40, 2002.

\bibitem{Ciarlet2}
 P. J. Ciarlet,  $\mathbb T$-coercivity: application to the discretization of Helmholtz-like problems,
 Computers and Mathematics with Applications, 64 (2012), 22-34.

\bibitem{ColtonKress}
D. Colton and R. Kress, Inverse Acoustic and Electromagnetic Scattering Theory, 2nd ed.,
Springer-Verlag, New York, 1998.

\bibitem{ColtonMonkSun}
D. Colton, P. Monk and J. Sun, Analytical and computational methods for transmission
eigenvalues, Inverse Probl., 26 (2010), 045011.

\bibitem{ColtonPaivarintaSykvester}
D. Colton, L. P\"{a}iv\"{a}rinta and J. Sylvester, The interior transmission problem,
 Inverse Probl. Imaging, 1 (2007), 13-28.


\bibitem{HsialLiuSunXu}
G. Hsiao, F. Liu,  J. Sun and L. Xu,
A coupled BEM and FEM for the interior transmission problem in acoustics,
J. Comput. Appl. Math., 235 (2011), 5213-5221.


\bibitem{JiSunTurner}
X. Ji, J. Sun and T. Turner, A mixed finite element method for Helmholtz Transmission eigenvalues,
ACM T. Math. Software, 38, Algorithm 922, 2012.

\bibitem{JiSunXie}
X. Ji, J. Sun and H. Xie, A multigrid method for Helmholtz transmission eigenvalue problems,
J. Sci. Comput., 60(2) (2014), 276-294.

\bibitem{Kirsch}
K. Kirsch, On the existence of transmission eigenvalues, Inverse Probl. Imaging, 3 (2009), 155-172.

\bibitem{LinXie_Steklov}
Q. Lin and H. Xie, A multilevel correction type of adaptive finite element
 method for Steklov eigenvalue problem,
 Proceedings of the International Conference Applications of Mathematics, 2012.

 \bibitem{LinXie}
Q. Lin and H. Xie, A multi-level correction scheme for eigenvalue problems,
Math. Comp., 84 (2015), 71-88.


\bibitem{PaivarintaSylvester}
L. P\"{a}iv\"{a}rinta and J. Sylvester, Transmission eigenvalues, SIAM J. Math. Anal.,
 40 (2008), 738-753.

\bibitem{Shaidurov}
V. Shaidurov, Multigrid methods for finite element, Kluwer
Academic Publics, Netherlands, 1995.

\bibitem{Sun_1}
J. Sun, Estimation of transmission eigenvalues and the index of refraction from Cauchy data,
Inverse Probl., 27 (2011), 015009.

\bibitem{Sun}
J. Sun, Iterative methods for transmission eigenvalues,  SIAM J. Numer. Anal., 49 (2011), 1860-1874.

\bibitem{WuChen}
X. Wu and W. Chen,
Error estimates of the finite element method for interior transmission problems,
J. Sci. Comput.,  57(2) (2013), 331-348.

\bibitem{Xie_Nonconforming}
H. Xie, A type of multi-level correction scheme for eigenvalue problems by nonconforming finite element methods,
BIT Numerical Mathematics, DOI. 10.1007/s10543-015-0545-1, 2015.

\bibitem{Xie_IMA}
H. Xie, A type of multilevel method for the Steklov eigenvalue problem,
IMA J. Numer. Anal., 34 (2014), 592-608.

\bibitem{Xie_JCP}
H. Xie, A multigrid method for eigenvalue problem,
J. Comput. Phys., 274 (2014),  550-561.

\bibitem{ZZ}
O. Zienkiewicz and J. Zhu, A simple error estimator and adaptive procedure
for practical engineering analysis, Int. J. Numer. Methods Eng., 24 (1987), 337-357.
\end{thebibliography}
\end{document}